\documentclass[12pt]{amsart}
\usepackage{amsmath,amssymb}
\usepackage[utf8,utf8x]{inputenc}
\usepackage[pagebackref=true,
pdftitle = {Simplicial\ arrangements\ with\ up\ to\ 27\ lines},
pdfkeywords = {arrangement\ of\ hyperplanes; simplicial; pseudoline; wiring;},
pdfsubject = {52C30; 20F55},
pdfauthor = {}]
{hyperref}
\usepackage{longtable}
\usepackage{graphicx}
\usepackage{pict2e}

\newtheorem{propo}{Proposition}[section]

\newtheorem{theor}[propo]{Theorem}
\newtheorem{lemma}[propo]{Lemma}
\newtheorem{conje}[propo]{Conjecture}
\newtheorem{algor}[propo]{Algorithm}

\theoremstyle{definition}
\newtheorem{defin}[propo]{Definition}
\newtheorem{examp}[propo]{Example}

\theoremstyle{remark}
\newtheorem{remar}[propo]{Remark}

\numberwithin{equation}{section}

\newcommand{\algo}[6]
{\vspace{11pt}\addtocounter{propo}{1}
\begin{algor}{{\bf #1}{\rm (#2)}}\label{#1}\end{algor}
\vspace{-6pt}\noindent{\it #3}.\\
{\bf Input:} #4\\
{\bf Output:} #5\\
\newcounter{#1}
\begin{list}{\textbf{\arabic{#1}.}}{\usecounter{#1}}
#6\end{list}\vspace{3pt}}

\newcommand{\NN }{\mathbb{N}}

\newcommand{\RR }{\mathbb{R}}

\newcommand{\PP }{\mathbb{P}}
\newcommand{\id }{\mathrm{id}}

\newcommand{\Ac }{\mathcal{A}}

\newcommand{\Dc }{\mathcal{D}}

\DeclareMathOperator{\Aut}{Aut}

\DeclareMathOperator{\Alt}{Alt}
\newcommand{\Spairs }{\Psi}

\title[Simplicial arrangements with up to 27 lines]
{Simplicial arrangements\\ with up to 27 lines}

\author{M.~Cuntz}
\address{Michael Cuntz,
Fachbereich Mathematik,
Universit\"at Kai\-sers\-lau\-tern,
Postfach 3049,
D-67653 Kaiserslautern, Germany}
\email{cuntz@mathematik.uni-kl.de}

\begin{document}

\begin{abstract}
We compute all isomorphism classes of simplicial arrangements in the
real projective plane with up to $27$ lines.
It turns out that Gr\"unbaums catalogue is complete up to $27$
lines except for four new arrangements with $22$, $23$, $24$, $25$ lines, respectively.
As a byproduct we classify simplicial arrangements of pseudolines
with up to $27$ lines. In particular, we disprove Gr\"unbaums conjecture
about unstretchable arrangements with at most $16$ lines,
and prove the conjecture that any simplicial arrangement with at most
$14$ pseudolines is stretchable.
\end{abstract}

\maketitle

\section{Introduction}

A simplicial arrangement is a finite set $\Ac=\{H_1,\ldots,H_n\}$ of
hyperplanes in $\RR^r$ such that the connected components of
$\RR^r\backslash\bigcup_{H\in\Ac}H$ are simplicial cones (compare \cite[Def.\ 5.14]{OT}).
For $r=3$, it may be visualized as a triangulation of
the plane by lines (for example Fig.\ \ref{a25}).
The classification of simplicial arrangements is
still an open problem. There is an impressive catalogue
\cite{p-G-09} of the known simplicial arrangements in dimension three: the
many sporadic cases are probably the main reason why a classification
appears to be so difficult.

But let us assume that we have some theorem stating that a given
catalogue is complete. Then it would be desirable to find a short proof.
However, it is possible that a shortest proof of such a theorem is very long,
so long that it would take several thousand pages to print it.
In such a situation we nowadays take advantage of a computer:
if the proof consists of a case-by-case
analysis, i.e.\ we have to comb through a large tree with branches mostly
leading to contradictions, then a computer is exactly the right tool to use.

In a previous work \cite{p-CH09c}, together with Heckenberger we have
classified the large class of \emph{crystallographic arrangements}
(see \cite{p-C10} for a definition).
The method used there is an algorithm which enumerates them all.
This algorithm terminates because it exploits several strong theorems.
But still, a computer has to check billions of branches of
the enumerating tree. We do not claim that a short proof does not exist,
but any other proof would need to address the large number
of sporadic arrangements arising here.

The algorithm we propose does not have polynomial runtime in the number
of lines of the arrangement. However, since it is conjectured that the largest
sporadic simplicial arrangement has $37$ lines, one idea for a classification could
be to determine all the simplicial arrangement with up to $37$ lines, and then prove
in some other way that any simplicial arrangement with more than
$37$ lines belongs to one of the infinite series.
With this plan in mind, we are talking about algorithms with constant runtime,
and thus we ``only'' need to make this constant small enough.

We should also note here that although we concentrate on the case
of dimension three (or the projective plane), I believe that one
can deduce a complete classification for all dimensions based upon a
classification in dimension three. This was the case for the crystallographic
arrangements mentioned above, see \cite{p-CH10}.

The most important result of our computation is:

\begin{theor}
We have a complete list of all simplicial arrangements in the real
projective plane with at most $27$ lines.
\end{theor}

We achieve this by enumerating wiring diagrams (or allowable sequences)
and using the correspondence of Goodman and Pollack to arrangements of pseudolines.
As a byproduct, we find further examples of simplicial arrangements
(of pseudolines) disproving the following two conjectures:

\begin{conje}{(\cite[Conj.\ 2]{p-G-09b})}
Up to isomorphism, there are only $5$ simplicial unstretchable arrangements of
$15$ or $16$ pseudolines.
\end{conje}

\begin{conje}{(\cite{p-G-09} or \cite[Conj.\ 3]{p-G-09b})}
There are no simplicial arrangements of straight lines besides the
three infinite families and $90$ sporadic arrangements listed in \cite{p-G-09}.
\end{conje}

We further obtain a proof for:

\begin{conje}{(\cite[Conj.\ 1]{b-G-72}, \cite[6.3]{BLVSWZ} or
\cite[3.1]{p-G-09b})}
All simplicial arrangements with at most $14$ pseudolines are stretchable.
\end{conje}

The new arrangements have $22$, $23$, $24$ and $25$ lines and are all contained
in the largest one with $25$ lines, see Fig.\ref{a25}.
To obtain the smaller arrangements, remove the thicker lines.

In view of the (to my eyes) bizarre structure of the Hasse diagram,
Fig.\ \ref{hassediag} or \cite[Fig.\ 4]{p-G-09}, and considering the
great amount of computations needed for our result,
it is very impressive that Gr\"unbaum missed only four arrangements.
Notice that there are a few connections missing in \cite[Fig.\ 4]{p-G-09}; in
particular it turns out that the arrangement $\Ac(21,4)$ is not maximal
but contained in the arrangement $\Ac(26,4)$.
We reproduce the table \cite[pp.\ 6--10]{p-G-09} here (up to $27$ lines)
because of the new arrangements and because of a new column containing
the automorphism groups of the arrangements.

This article is organized as follows: We first review some properties
of wirings in general. After several preparations we then give an algorithm to
enumerate simplicial wirings. Finally, we discuss stretchability and
summarize the results.

\begin{figure}
\begin{center}
\setlength{\unitlength}{0.75pt}
\begin{picture}(400,400)(100,200)
\moveto(483.279413355678638377155548610,480.054085717084338934572008031)
\lineto(130.503014223155076370906651531,293.836108715985595420223336931)
\moveto(483.279413366770748824747498223,319.945914308310436083902262571)
\lineto(130.503014233221380939685096682,506.163891300085860340335525427)
\moveto(489.736659610102759919933612666,336.754446796632413360022129111)
\lineto(110.263340389897240080066387334,336.754446796632413360022129111)
\moveto(445.954063266690261424521593589,536.738478183512692236918153419)
\lineto(142.115774818872439595591226237,277.229598685371212288294142057)
\moveto(499.968941388255175537018828007,396.475446118929904989386784716)
\lineto(100.031058611744824462981171993,396.475446118929904989386784716)
\moveto(500.000000000000000000000000000,400.000000000000000000000000000)
\lineto(100.000000000000000000000000000,400.000000000000000000000000000)
\moveto(412.292242319575928957559574424,565.500611221957850490508954421)
\lineto(177.908477497051516354885389898,241.590214529177446918065556730)
\moveto(460.513064322353212583522693744,280.687149971815272993458539764)
\lineto(157.049756033705141617363954623,539.875758264170923619199083957)
\strokepath
\linethickness{1.3pt}
\put(100.558059476087039337223223320,414.930249832335867314618844439)
{\line(1,0){398.883881047825921325553553360}}
\strokepath
\thinlines
\moveto(445.954063246347847139173955770,263.261521794773902285703895248)
\lineto(142.115774800608038833929810232,522.770401291140546798449617421)
\moveto(471.134416669537106401377383759,296.496321653879082232897979214)
\lineto(117.993254917023979504547325809,482.906843772396643963738598479)
\moveto(460.513064322353212583522693744,519.312850028184727006541460236)
\lineto(157.049756033705141617363954623,260.124241735829076380800916043)
\moveto(484.441096960441932830799310458,322.659960227390252484334427667)
\lineto(116.588784385253314218725956535,479.751651937253948704511175819)
\strokepath
\linethickness{1.3pt}
\put(303.338505347899997175593871336,200.027865986177856364144395549)
{\line(0,1){399.944268027644287271711208902}}
\strokepath
\thinlines
\moveto(426.334971743568835670986142977,244.953313758237116378416192356)
\lineto(192.229050979468058777861179225,568.479739278092145968955224612)
\moveto(471.134416683226558598751035871,503.503678323486587720988154081)
\lineto(117.993254927989283280041109110,317.093156203531045984840639524)
\moveto(338.299787310127281517364357395,596.298564161837447847100721414)
\lineto(338.299787310127281517364357395,203.701435838162552152899278586)
\strokepath
\linethickness{1.3pt}
\put(100.558059476087039337223223320,385.069750167664132685381155561)
{\line(1,0){398.883881047825921325553553360}}
\strokepath
\thinlines
\moveto(489.736659610102759919933612666,463.245553203367586639977870889)
\lineto(110.263340389897240080066387334,463.245553203367586639977870889)
\moveto(499.968941388255175537018828007,403.524553881070095010613215284)
\lineto(100.031058611744824462981171993,403.524553881070095010613215284)
\moveto(268.377223398316206680011064556,597.484176581314990174384610437)
\lineto(268.377223398316206680011064556,202.515823418685009825615389563)
\moveto(412.292242345792395890047456540,234.499388795830034619472747994)
\lineto(177.908477522144746358038505178,558.409785490162713451409036924)
\moveto(484.441096956146259179388725696,477.340039782854101704233379992)
\lineto(116.588784380823693147337672844,320.248348072933203106675444791)
\moveto(426.334971754565657189570080318,555.046686232802465507560496471)
\lineto(192.229050991417631305617993376,231.520260714264105428847851383)
\strokepath
\put(470,560){\Large $\infty$}
\end{picture}
\end{center}
\caption{A new simplicial arrangement with $25$ lines\label{a25}.
The symbol ``$\infty$'' stands for the line at infinity; removing the thicker lines
gives further new simplicial arrangements with $22$, $23$ and $24$ lines, respectively.}
\end{figure}
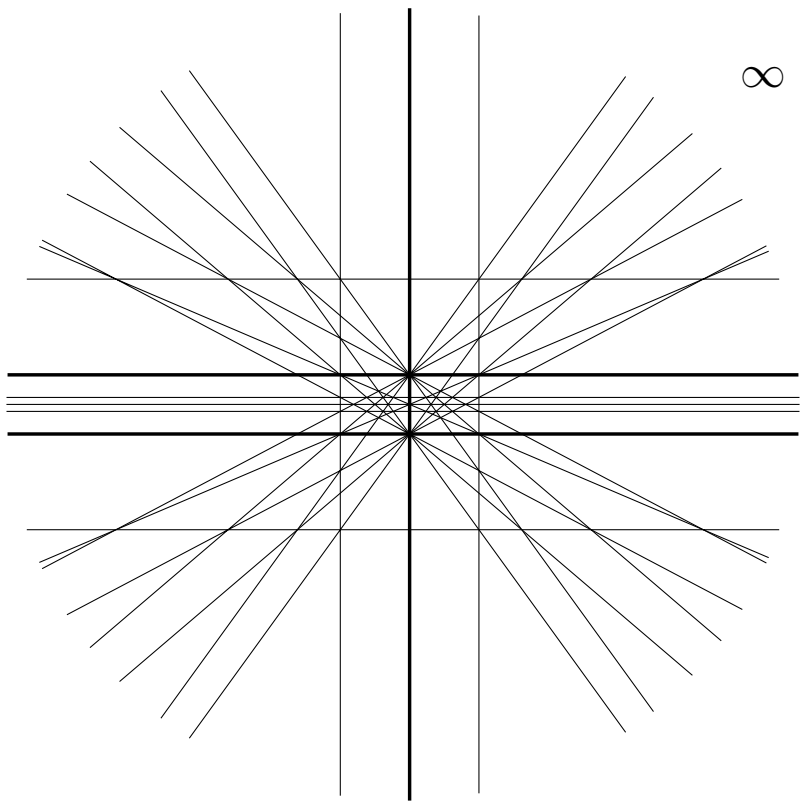

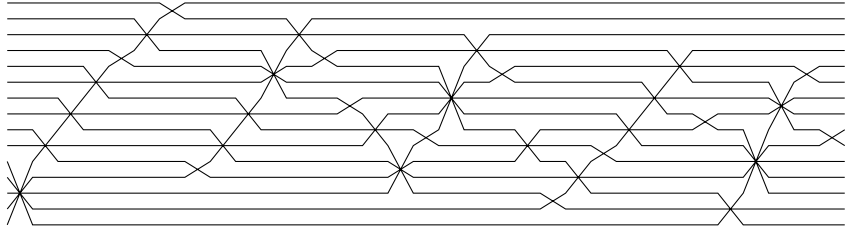
\begin{figure}
\begin{center}
\setlength{\unitlength}{1.2pt}
\begin{picture}(400,60)(40,310)
\moveto(58,305)\lineto(58,305)\lineto(66,325)\lineto(66,325)\lineto(74,335)\lineto(74,335)\lineto(82,345)\lineto(82,345)\lineto(90,355)\lineto(90,355)\lineto(98,360)\lineto(98,360)\lineto(106,370)\lineto(106,370)\lineto(114,375)\lineto(322,375)
\moveto(58,310)\lineto(58,310)\lineto(66,320)\lineto(114,320)\lineto(122,325)\lineto(122,325)\lineto(130,335)\lineto(130,335)\lineto(138,345)\lineto(138,345)\lineto(146,360)\lineto(146,360)\lineto(154,370)\lineto(322,370)
\moveto(58,315)\lineto(58,315)\lineto(66,315)\lineto(178,315)\lineto(186,330)\lineto(186,330)\lineto(194,335)\lineto(194,335)\lineto(202,355)\lineto(202,355)\lineto(210,365)\lineto(322,365)
\moveto(58,320)\lineto(58,320)\lineto(66,310)\lineto(226,310)\lineto(234,315)\lineto(234,315)\lineto(242,325)\lineto(242,325)\lineto(250,330)\lineto(250,330)\lineto(258,340)\lineto(258,340)\lineto(266,350)\lineto(266,350)\lineto(274,360)\lineto(322,360)
\moveto(58,325)\lineto(58,325)\lineto(66,305)\lineto(282,305)\lineto(290,315)\lineto(290,315)\lineto(298,335)\lineto(298,335)\lineto(306,350)\lineto(306,350)\lineto(314,355)\lineto(322,355)
\moveto(58,330)\lineto(66,330)\lineto(74,330)\lineto(122,330)\lineto(130,330)\lineto(170,330)\lineto(178,340)\lineto(194,340)\lineto(202,350)\lineto(210,350)\lineto(218,355)\lineto(266,355)\lineto(274,355)\lineto(306,355)\lineto(314,350)\lineto(322,350)
\moveto(58,335)\lineto(66,335)\lineto(74,325)\lineto(114,325)\lineto(122,320)\lineto(178,320)\lineto(186,325)\lineto(218,325)\lineto(226,335)\lineto(250,335)\lineto(258,335)\lineto(274,335)\lineto(282,340)\lineto(298,340)\lineto(306,345)\lineto(322,345)
\moveto(58,340)\lineto(74,340)\lineto(82,340)\lineto(130,340)\lineto(138,340)\lineto(162,340)\lineto(170,345)\lineto(194,345)\lineto(202,345)\lineto(258,345)\lineto(266,345)\lineto(298,345)\lineto(306,340)\lineto(322,340)
\moveto(58,345)\lineto(74,345)\lineto(82,335)\lineto(122,335)\lineto(130,325)\lineto(178,325)\lineto(186,320)\lineto(234,320)\lineto(242,320)\lineto(290,320)\lineto(298,330)\lineto(314,330)\lineto(322,335)\lineto(322,335)
\moveto(58,350)\lineto(82,350)\lineto(90,350)\lineto(138,350)\lineto(146,355)\lineto(154,355)\lineto(162,360)\lineto(202,360)\lineto(210,360)\lineto(266,360)\lineto(274,350)\lineto(298,350)\lineto(306,335)\lineto(314,335)\lineto(322,330)\lineto(322,330)
\moveto(58,355)\lineto(82,355)\lineto(90,345)\lineto(130,345)\lineto(138,335)\lineto(170,335)\lineto(178,335)\lineto(186,335)\lineto(194,330)\lineto(218,330)\lineto(226,330)\lineto(242,330)\lineto(250,325)\lineto(290,325)\lineto(298,325)\lineto(322,325)
\moveto(58,360)\lineto(90,360)\lineto(98,355)\lineto(138,355)\lineto(146,350)\lineto(194,350)\lineto(202,340)\lineto(250,340)\lineto(258,330)\lineto(290,330)\lineto(298,320)\lineto(322,320)
\moveto(58,365)\lineto(98,365)\lineto(106,365)\lineto(146,365)\lineto(154,365)\lineto(202,365)\lineto(210,355)\lineto(210,355)\lineto(218,350)\lineto(258,350)\lineto(266,340)\lineto(274,340)\lineto(282,335)\lineto(290,335)\lineto(298,315)\lineto(322,315)
\moveto(58,370)\lineto(98,370)\lineto(106,360)\lineto(138,360)\lineto(146,345)\lineto(162,345)\lineto(170,340)\lineto(170,340)\lineto(178,330)\lineto(178,330)\lineto(186,315)\lineto(226,315)\lineto(234,310)\lineto(282,310)\lineto(290,310)\lineto(322,310)
\moveto(58,375)\lineto(106,375)\lineto(114,370)\lineto(146,370)\lineto(154,360)\lineto(154,360)\lineto(162,355)\lineto(194,355)\lineto(202,335)\lineto(218,335)\lineto(226,325)\lineto(234,325)\lineto(242,315)\lineto(282,315)\lineto(290,305)\lineto(322,305)
\strokepath
\end{picture}
\end{center}
\caption{An unstretchable wiring with $15$ lines\label{a15_ns}}
\end{figure}
\vspace{\baselineskip}
\textbf{Acknowledgement.}
I would like to thank M.\ Barakat for many valuable comments on a
previous version of the manuscript.

\section{Wiring diagrams}

We first recall some definitions (compare \cite[6.4]{BLVSWZ}).

\begin{defin}\label{allowseq}
A sequence $\Sigma=(\sigma_0,\ldots,\sigma_m)$ of permutations in $S_n$ is an
\emph{allowable sequence} if
\begin{enumerate}
\item $\sigma_0=\id_{S_n}=[1,\ldots,n]$,
\item $\sigma_m=(1\: n)(2\: (n-1))\ldots=[n,\ldots,1]$,
\item \label{allowseq_3} for each $i=0,\ldots,m-1$ there are $1\le a< b\le n$ such that
\begin{eqnarray*}
& & \sigma_i(a)< \sigma_i(a+1)< \ldots < \sigma_i(b),\\
& & \sigma_{i+1}(c)=\sigma_i(a+b-c) \quad\text{for}\quad a\le c \le b,\\
& & \sigma_{i+1}(c)=\sigma_i(c) \quad\text{for}\quad c<a \text{ or } c>b.\\
\end{eqnarray*}
\end{enumerate}
We will encode an allowable sequence by the sequence of $a$'s and $b$'s, i.e.,
$\Sigma$ is uniquely determined by the sequence of pairs
\[ \Spairs=((a_1,b_1),(a_2,b_2),\ldots,(a_m,b_m)) \]
with
\[ \sigma_{i+1}=(\sigma_i(a_{i+1})\:\sigma_i(b_{i+1}))(\sigma_i(a_{i+1}+1)\:\sigma_i(b_{i+1}-1))\ldots \sigma_i.\]
\end{defin}

Allowable sequences are in one-to-one correspondence to \emph{wiring diagrams}
(for example Fig.\ \ref{a15_ns}).
These are a very useful representation for arrangements of pseudolines
(for a proof see \cite[Thm.\ 2.9]{p-GP84} or \cite[Thm.\ 6.3.3]{BLVSWZ}):
\begin{theor}[Goodman and Pollack]
Every arrangement of pseudolines is isomorphic to a wiring diagram arrangement.
\end{theor}
\begin{remar}
It is easy to give an algorithm which computes a wiring for a given arrangement
of straight lines, see \cite[Sect.\ 2]{p-GP84} or Lemma \ref{wirbeg} for more details.
\end{remar}
Figure \ref{a15_ns} shows an example of an \emph{unstretchable} wiring,
i.e., the cell decomposition of $\PP^2$ induced by the wiring is not
combinatorially isomorphic to the cell decomposition induced by some arrangement
of straight lines. Notice that this example is \emph{simplicial} which means
that all $2$-cells have exactly $3$ vertices. We will see that there is no
unstretchable simplicial wiring with less than $15$ lines.

Our goal is to design an algorithm to enumerate simplicial arrangements,
or more generally simplicial arrangements of pseudolines.
By the above theorem, we may enumerate certain wiring diagrams instead.
However, there are many different wiring diagrams which yield isomorphic
arrangements. So the most important part will be to recognize symmetries
to avoid computations producing no ``new'' wirings.

But let us first look at a naive version of such an algorithm.
Let $W_n$ denote the set of allowable sequences in $S_n$ (equivalently
the set of wiring diagrams with $n$ rows). During the algorithm, we will
successively enlarge a sequence $\Spairs$ until it becomes an element of $W_n$.
More precisely, we encode an allowable sequence in construction by the
following data.
\begin{defin}
A \emph{wiring fragment} $\omega$
consists of:
\begin{itemize}
\item $m\in\NN,$
\item $\sigma_m=[\sigma_m(1),\ldots,\sigma_m(n)]\in \{1,\ldots,n\}^n$,
\item $\Spairs=((a_1,b_1),(a_2,b_2),\ldots,(a_m,b_m))$,
\item Number of lines $\varepsilon_m^i$ going through the last junction in row $i$ for $i=1,\ldots,n$:
$\varepsilon_m^i:=b_k-a_k+1$ where $k$ is maximal with $a_k\le i\le b_k$.
\item For each $i$, the number $v_m^i$ of finished vertices of the last $2$-cell between
row $i$ and $i+1$
(let $v_m^0$ be the number of vertices of the polygon between row $n$ and $1$),
\item For each $i$, the number $s_i$ of finished vertices of the first $2$-cell between
row $i$ and $i+1$ (we define $s_0:=0$),
\item The maximal $n\ge d_m\in\NN$ such that $\sigma_m(i)=n+1-i$ for all $i=1,\ldots,d_m-1$,
\item The minimal $1\le u_m\in\NN$ such that $\sigma_m(i)=n+1-i$ for all $i=u_m+1,\ldots,n$,
\end{itemize}
where $\sigma_m$ and $\Spairs$ are being interpreted as the beginning of an
allowable sequence and have to satisfy the corresponding axioms.

We will call $\omega$ \emph{complete} if $\sigma_m=[n,\ldots,1]$. In this case,
$\omega$ ``is'' an allowable sequence (or a wiring), and $d_m\ge u_m$.
\end{defin}
The wiring fragment is continuously updated during the algorithm.
Note that we will need most variables of the wiring fragment later but mention them
already in this section to avoid a second definition.

\begin{examp}
Fig.\ \ref{wirfragex} shows a wiring fragment with the following data:
\begin{tiny}
\begin{eqnarray*}
\sigma_m &=& [5,4,11,8,13,9,6,16,10,19,14,20,12,7,17,18,15,3,2,1],\\
\Spairs &=& ((1,5),(5,6),(6,8),(8,9),(9,11),(11,13),(13,14),(14,16),\\
&&(16,17),(17,19),(19,20),(4,6),(6,9),(9,11),(11,14),(14,17),\\
&&(17,19),(16,17),(8,9),(5,6),(3,5),(5,8),(8,11),(11,12),\\
&&(12,14),(14,16),(16,18),(10,12),(12,14)),\\
s_i&=&0,1,1,1,1,2,2,1,2,2,1,2,1,2,2,1,2,2,1,2,\\
v^i_m&=&2,1,2,1,2,1,1,2,1,2,1,2,1,1,2,2,1,1,2,2,\\
\varepsilon^i_m&=&5,5,3,3,4,4,4,4,4,3,3,3,3,3,3,3,3,3,3,2,\\
d_m &=& 1,\\ u_m &=& 17.
\end{eqnarray*}
\end{tiny}
\begin{figure}
\begin{center}
\setlength{\unitlength}{1.2pt}
\begin{picture}(400,150)(20,310)
\put(50,306){\tiny $ 1$}\moveto(58,308)\lineto(58,308)\lineto(66,340)\lineto(66,340)\lineto(74,348)\lineto(74,348)\lineto(82,364)\lineto(82,364)\lineto(90,372)\lineto(90,372)\lineto(98,388)\lineto(98,388)\lineto(106,404)\lineto(106,404)\lineto(114,412)\lineto(114,412)\lineto(122,428)\lineto(122,428)\lineto(130,436)\lineto(130,436)\lineto(138,452)\lineto(138,452)\lineto(146,460)\lineto(290,460)\put(290,458){\tiny $ 1$}
\put(50,314){\tiny $ 2$}\moveto(58,316)\lineto(58,316)\lineto(66,332)\lineto(146,332)\lineto(154,348)\lineto(154,348)\lineto(162,372)\lineto(162,372)\lineto(170,388)\lineto(170,388)\lineto(178,412)\lineto(178,412)\lineto(186,436)\lineto(186,436)\lineto(194,452)\lineto(290,452)\put(290,450){\tiny $ 2$}
\put(50,322){\tiny $ 3$}\moveto(58,324)\lineto(58,324)\lineto(66,324)\lineto(218,324)\lineto(226,340)\lineto(226,340)\lineto(234,364)\lineto(234,364)\lineto(242,388)\lineto(242,388)\lineto(250,396)\lineto(250,396)\lineto(258,412)\lineto(258,412)\lineto(266,428)\lineto(266,428)\lineto(274,444)\lineto(290,444)\put(290,442){\tiny $ 3$}
\put(50,330){\tiny $ 4$}\moveto(58,332)\lineto(58,332)\lineto(66,316)\lineto(290,316)\put(290,314){\tiny $ 4$}
\put(50,338){\tiny $ 5$}\moveto(58,340)\lineto(58,340)\lineto(66,308)\lineto(290,308)\put(290,306){\tiny $ 5$}
\put(50,346){\tiny $ 6$}\moveto(58,348)\lineto(66,348)\lineto(74,340)\lineto(146,340)\lineto(154,340)\lineto(210,340)\lineto(218,348)\lineto(226,348)\lineto(234,356)\lineto(290,356)\put(290,354){\tiny $ 6$}
\put(50,354){\tiny $ 7$}\moveto(58,356)\lineto(74,356)\lineto(82,356)\lineto(154,356)\lineto(162,364)\lineto(202,364)\lineto(210,372)\lineto(234,372)\lineto(242,380)\lineto(274,380)\lineto(282,396)\lineto(282,396)\lineto(290,412)\lineto(290,412)\put(290,410){\tiny $ 7$}
\put(50,362){\tiny $ 8$}\moveto(58,364)\lineto(74,364)\lineto(82,348)\lineto(146,348)\lineto(154,332)\lineto(218,332)\lineto(226,332)\lineto(290,332)\put(290,330){\tiny $ 8$}
\put(50,370){\tiny $ 9$}\moveto(58,372)\lineto(82,372)\lineto(90,364)\lineto(154,364)\lineto(162,356)\lineto(226,356)\lineto(234,348)\lineto(290,348)\put(290,346){\tiny $ 9$}
\put(50,378){\tiny $10$}\moveto(58,380)\lineto(90,380)\lineto(98,380)\lineto(162,380)\lineto(170,380)\lineto(234,380)\lineto(242,372)\lineto(290,372)\put(290,370){\tiny $10$}
\put(50,386){\tiny $11$}\moveto(58,388)\lineto(90,388)\lineto(98,372)\lineto(154,372)\lineto(162,348)\lineto(210,348)\lineto(218,340)\lineto(218,340)\lineto(226,324)\lineto(290,324)\put(290,322){\tiny $11$}
\put(50,394){\tiny $12$}\moveto(58,396)\lineto(98,396)\lineto(106,396)\lineto(170,396)\lineto(178,404)\lineto(250,404)\lineto(258,404)\lineto(282,404)\lineto(290,404)\lineto(290,404)\put(290,402){\tiny $12$}
\put(50,402){\tiny $13$}\moveto(58,404)\lineto(98,404)\lineto(106,388)\lineto(162,388)\lineto(170,372)\lineto(202,372)\lineto(210,364)\lineto(226,364)\lineto(234,340)\lineto(290,340)\put(290,338){\tiny $13$}
\put(50,410){\tiny $14$}\moveto(58,412)\lineto(106,412)\lineto(114,404)\lineto(170,404)\lineto(178,396)\lineto(242,396)\lineto(250,388)\lineto(274,388)\lineto(282,388)\lineto(290,388)\put(290,386){\tiny $14$}
\put(50,418){\tiny $15$}\moveto(58,420)\lineto(114,420)\lineto(122,420)\lineto(178,420)\lineto(186,428)\lineto(194,428)\lineto(202,436)\lineto(266,436)\lineto(274,436)\lineto(290,436)\put(290,434){\tiny $15$}
\put(50,426){\tiny $16$}\moveto(58,428)\lineto(114,428)\lineto(122,412)\lineto(170,412)\lineto(178,388)\lineto(234,388)\lineto(242,364)\lineto(290,364)\put(290,362){\tiny $16$}
\put(50,434){\tiny $17$}\moveto(58,436)\lineto(122,436)\lineto(130,428)\lineto(178,428)\lineto(186,420)\lineto(258,420)\lineto(266,420)\lineto(290,420)\put(290,418){\tiny $17$}
\put(50,442){\tiny $18$}\moveto(58,444)\lineto(130,444)\lineto(138,444)\lineto(186,444)\lineto(194,444)\lineto(266,444)\lineto(274,428)\lineto(290,428)\put(290,426){\tiny $18$}
\put(50,450){\tiny $19$}\moveto(58,452)\lineto(130,452)\lineto(138,436)\lineto(178,436)\lineto(186,412)\lineto(250,412)\lineto(258,396)\lineto(274,396)\lineto(282,380)\lineto(290,380)\put(290,378){\tiny $19$}
\put(50,458){\tiny $20$}\moveto(58,460)\lineto(138,460)\lineto(146,452)\lineto(186,452)\lineto(194,436)\lineto(194,436)\lineto(202,428)\lineto(258,428)\lineto(266,412)\lineto(282,412)\lineto(290,396)\lineto(290,396)\put(290,394){\tiny $20$}
\strokepath
\end{picture}
\end{center}
\caption{A wiring fragment\label{wirfragex}}
\end{figure}
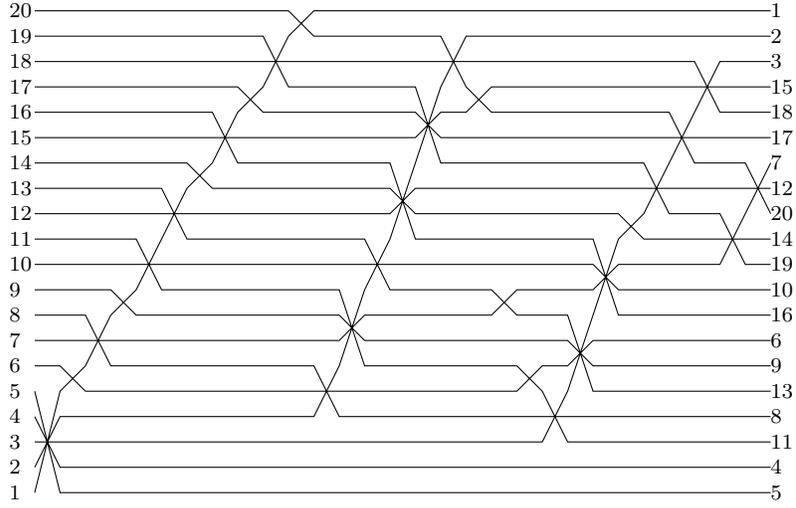
\end{examp}

Calling the following function with an initial ``empty'' wiring fragment $\omega$ will
enumerate all wirings with $n$ rows:

\algo{EnumerateWirings}{$\omega$}{Enumerates all allowable sequences in $W_n$}
{A wiring fragment $\omega$}{List of completions of the fragment $\omega$ to a wiring}
{
\item If $\sigma_m=[n,\ldots,1]$ then return $\{\omega\}$.
\item $R\leftarrow \emptyset$.
\item For all $1\le i<j\le n$ do:
\item If $\sigma_m(i)<\sigma_m(i+1)<\ldots<\sigma_m(j-1)<\sigma_m(j)$, then
update all the data to a new fragment $\omega'$ with $\Spairs'(m+1)=(i,j)$,
and call $R:=R\:\cup\:${\bf EnumerateWirings}$(\omega')$.
\item Return $R$.
}

Of course, this algorithm will not lead very far. But
let $L_n$ be the set of isomorphism classes of arrangements with $n$ pseudolines.
We have a surjective map
\[ \pi : W_n \rightarrow L_n, \]
mapping a wiring to an arrangement up to isomorphisms.
The most important improvements to our algorithm will be to find a smaller set $W_n'\subseteq W_n$
such that $\pi|_{W'_n}$ is still surjective.

\begin{lemma}
Let $W'_n$ be the set of allowable sequences
\[ \Spairs=((a_1,b_1),(a_2,b_2),\ldots,(a_m,b_m)) \]
such that for each $1\le i<m$ we have
\[ b_i=a_{i+1} \quad \text{or} \quad b_{i+1}\le a_i. \]
Then $\pi|_{W'_n}$ is surjective.
\end{lemma}
\begin{proof}
Observe that if $\Spairs(i)=(a_i,b_i)$, $\Spairs(i+1)=(a_{i+1},b_{i+1})$, and
either $a_i>b_{i+1}$ or $a_{i+1}>b_i$, then we may interchange $\Spairs(i)$ and $\Spairs(i+1)$
without changing the image under $\pi$. Moreover, the intervals $[a_i,b_i]$ and $[a_{i+1},b_{i+1}]$
may intersect in only one point by Def.\ \ref{allowseq} (\ref{allowseq_3}).
\end{proof}

This gives a new version of the algorithm:

\algo{EnumerateWirings2}{$\omega$}{Enumerates all allowable sequences in $W'_n$}
{A wiring fragment $\omega$}{List of completions of the fragment $\omega$}
{
\item If $\sigma_m=[n,\ldots,1]$, then return $\{\omega\}$.
\item $R\leftarrow \emptyset$.
\item For all $i = b_m+1,\ldots,n$ do:
\item If $\sigma_m(b_m)<\sigma_m(b_m+1)<\ldots<\sigma_m(i)$, then
update all the data to a new fragment $\omega'$ with $\Spairs'(m+1)=(b_m,i)$,
and call $R:=R\:\cup\:${\bf EnumerateWirings2}$(\omega')$.
\item For all $1\le i<j\le a_m$ do:
\item If $\sigma_m(i)<\sigma_m(i+1)<\ldots<\sigma_m(j-1)<\sigma_m(j)$, then
update all the data to a new fragment $\omega'$ with $\Spairs'(m+1)=(i,j)$,
and call $R:=R\:\cup\:${\bf EnumerateWirings2}$(\omega')$.
\item Return $R$.
}

Further, we may reduce the symmetries by requiring a certain beginning:

\begin{lemma}\label{wirbeg}
Let $W''_n$ be the set of allowable sequences
\[ \Spairs=((a_1,b_1),(a_2,b_2),\ldots,(a_m,b_m)) \]
such that there exists an $m_0$ with
\[ a_1=1, \quad b_i=a_{i+1}, \quad b_{m_0}=n, \]
for $i=1,\ldots,m_0-1$. Then $\pi|_{W''_n}$ is surjective.
\end{lemma}
\begin{proof}
This holds by \cite[Lemma 2.4]{p-GP84}.
We sketch here a construction: For a given arrangement $\Ac$,
choose a line $H\in\Ac$. Modifying this line slightly we may assume
that we have a line $H'$ on which no vertex of the arrangement lies.
After the choice of a line at infinity, we rotate the line $H'$ and
pass each vertex exactly once. By choosing a line $H'$ close enough
to $H$, we ensure that the vertices on $H$ are passed first.
\end{proof}

So we call {\bf EnumerateWirings2} with all possible fragments
$\omega$ with $m_0$ junctions as in Lemma \ref{wirbeg} instead of
the empty fragment. Moreover, these fragments are uniquely determined
by the sequences
\[ (b_1-a_1,\ldots,b_{m_0}-a_{m_0}).\]
Considering that the procedure in the proof of Lemma \ref{wirbeg} depends on the choice
of a line at infinity and on the orientation, we obtain:

\begin{lemma}\label{wirbeg_dihedral}
It suffices to start the algorithm with one representative of the orbit
under the action of the dihedral group $\Dc_{m_0}$ on the sequence
$(b_1-a_1,\ldots,b_{m_0}-a_{m_0})$.
\end{lemma}

\begin{remar}\label{goodchoice}
It is quite important to choose a ``good'' representative $\sigma\in\Dc_{m_0}$
for the algorithm. The following is a choice that has proved to be best to average:
Let $(p_1,\ldots,p_{m_0})=(b_1-a_1,\ldots,b_{m_0}-a_{m_0})$. Then
choose a $\sigma\in \Dc_{m_0}$ such that 
$p_{\sigma(1)}=1$ and $p_{\sigma(m_0)}=1$, or choose a lexicographically
greatest $(p_{\sigma(1)},\ldots,p_{\sigma(m_0)})$, $\sigma\in \Dc_{m_0}$ with
$p_{\sigma(1)}=1$ or $p_{\sigma(m_0)}=1$.
\end{remar}

\begin{remar}\label{maxjunction}
For a vertex $v$, let $\ell(v)$ be the number of lines incident with $v$.
Without loss of generality we may assume that the line $H$ chosen in
Lemma \ref{wirbeg} contains a vertex $v$ with maximal $\ell(v)$.
This means that we only need to consider junctions of at most $\ell(v)$
lines during the algorithm.
\end{remar}

\section{Simplicial wirings}

\begin{defin}
A \emph{simplicial} wiring is a wiring in which all $2$-cells have exactly $3$ vertices.
A \emph{near pencil} consists of $n-1$ lines having one point in common and one further
line which is not incident with that point.
\end{defin}
Simplicial wirings are much easier to enumerate than arbitrary wir\-ings
because simpliciality gives many break conditions for the algorithm.
The easiest one is given by the following Lemma which should be folklore:
\begin{lemma}\label{twosquares}
If a simplicial wiring has two neighboring ordinary vertices (intersection points
where exactly two pseudolines meet), then it is a near pencil arrangement.
\end{lemma}
So since the near pencil arrangements may be ignored without loss of generality,
we can stop the enumeration when two neighboring vertices are ordinary.

The following trivial relation has more applications than expected:

\begin{lemma}\label{joinends}
If a simplicial wiring fragment is complete, i.e.\ $\sigma_m=[n,\ldots,1]$,
then $v_m^i+s_{n-i}=3$ for all $i=0,\ldots,n$.
\end{lemma}

A further important improvement is:

\begin{lemma}\label{all2}
If we only enumerate simplicial wirings with Alg.\ \ref{EnumerateWirings2},
then in steps $4$ and $6$ it suffices to consider new pairs $(k,\ell)$ such that
$v_m^{k-1}<2$ and $v_m^\ell<2$.
\end{lemma}
\begin{proof}
Otherwise $v_m^{i-1}$ resp.\ $v_m^j$ would get greater than $2$
after the next move, contradicting simpliciality.
\end{proof}

Here are some more conditions:

\begin{lemma}\label{simpobstr_1}
In Algorithm \ref{EnumerateWirings2}, arriving at junction $k$ the enumeration will never
yield a simplicial wiring if one of the following conditions is satisfied:
\begin{enumerate}
\item\label{lem1} $a_k\ge 2$ and $\sigma_k(a_k)<\sigma_k(a_k-1)$ and
($\sigma_k(a_k-1)\ne n+2-a_k$ or $\sigma_k(a_k)\ne n+1-a_k$),
\item\label{lem1b} $b_k<n$ and $\sigma_k(b_k+1)<\sigma_k(b_k)$ and
($\sigma_k(b_k-1)\ne n+1-b_k$ or $\sigma_k(b_k+1)\ne n-b_k$).
\end{enumerate}
\end{lemma}
\begin{proof}
Notice first that since the last move concerned the lines in the rows $a_k,\ldots,b_k$,
the numbers of finished vertices $v_k^i$ between the rows $i$ and $i+1$ for
$a_k\le i<b_k$ are all equal to $1$ whereas $v_k^{a_k-1}=2=v_k^{b_k}$.

Now if $\sigma_k(a_k)<\sigma_k(a_k-1)$, then the rows $k$ and $k-1$ will
not be moved anymore since otherwise $v_k^{a_k-1}>2$ which is contradicting the
simpliciality. Thus in this case we know that 
$\sigma_k(a_k-1)=n+2-a_k$ and $\sigma_k(a_k)=n+1-a_k$,
hence (\ref{lem1}). The case (\ref{lem1b}) is similar.
\end{proof}

\begin{lemma}\label{simpobstr_2}
In Algorithm \ref{EnumerateWirings2}, arriving at junction $k$ the enumeration will never
yield a simplicial wiring if one of the following conditions is satisfied:
\begin{enumerate}
\item there is a $j\in\{a_k,\ldots,b_k\}$ with
$\sigma_k(j)=\sigma_k(j+1)+1$ and $\sigma_k(j)\ne n+1-j$ and $s_{\sigma_k(j+1)}=2$,
\item there is a $j\in\{a_k+1,\ldots,b_k\}$ with
$\sigma_k(j)=\sigma_k(j+1)+1$ and $\sigma_k(j-1)\ne \sigma_k(j)+1$ and $s_{\sigma_k(j)}=2$
and $s_{\sigma_k(j+1)}=2$,
\item there is a $j\in\{a_k,\ldots,b_k-1\}$ with
$\sigma_k(j)=\sigma_k(j+1)+1$ and $\sigma_k(j+2)\ne \sigma_k(j)-2$ and $s_{\sigma_k(j)-2}=2$
and $s_{\sigma_k(j+1)}=2$.
\end{enumerate}
\end{lemma}
\begin{proof}
We prove (1); (2) and (3) are similar.
If $j\in\{a_k,\ldots,b_k\}$ and $\sigma_k(j)=\sigma_k(j+1)+1$, then
either the labels $\sigma_k(j)$ and $\sigma_k(j+1)$ are at their terminal position in
which case $\sigma_k(j)=n+1-j$, or they are not, but then the cell that these
labels will enclose at the end will have at least $2$ vertices which implies
$s_{\sigma_k(j+1)}=1$ by Lemma \ref{joinends}.
\end{proof}

\begin{lemma}\label{simpobstr_5}
In Algorithm \ref{EnumerateWirings2}, arriving at junction $k$ the enumeration will never
yield a simplicial wiring if one of the following conditions is satisfied:
\begin{enumerate}
\item $s_{n-u_k}+v_k^{u_k}=3$,
\item $d_k>1$ and $s_{n-1-d_k}+v_k^{d_k-1}=3$.
\end{enumerate}
\end{lemma}
\begin{proof}
The cells at the end of the rows $(u_k,u_k+1)$ and $(d_k-1,d_k)$ are not finished yet,
future change will increase $v_k^{u_k}$ resp.\ $v_k^{d_k-1}$,
contradicting Lemma \ref{joinends}.
\end{proof}

\begin{lemma}\label{simpobstr_4}
In Algorithm \ref{EnumerateWirings2}, arriving at junction $k$ the enumeration will never
yield a simplicial wiring if one of the following conditions is satisfied:
\begin{enumerate}
\item $\sigma_k(b_k)>\sigma(b_k+1)$ and $b_k\le u_k$,
\item $d_k<a_k\le u_k$ and $\sigma_k(a_k-1)>\sigma(a_k)$ and
($v_k(n+1-a_k)\ne 1$ or $\sigma_k(d_k)<\sigma_k(a_k)$ or $\sigma_k(a_k-2)<\sigma_k(a_k-1)$
or there is a $j\in\{ b_k+1,\ldots,u_k\}$ with $\sigma_k(j)>\sigma_k(a_k)$).
\end{enumerate}
\end{lemma}
\begin{proof}
(1) If $\sigma_k(b_k)>\sigma(b_k+1)$ and $b_k\le u_k$ then by
Lemma \ref{simpobstr_1} (\ref{lem1b}), the rows $b_k$ and $b_k+1$ will not
change anymore. Thus Alg.\ \ref{EnumerateWirings2} would only perform moves
$(i,j)$ with $i<j\le a_k$ from now on. This would never complete the fragment
since $b_k\le u_k$.

(2) If $d_k<a_k\le u_k$ and $\sigma_k(a_k-1)>\sigma(a_k)$ then the situation is
slightly different because Alg.\ \ref{EnumerateWirings2} also jumps past $a_k$.
But what we know is that the rows $a_k-1$ and $a_k$ are finished, so all future
moves will take place above or below $a_k$. This means that the fragment can
only be completed if $\sigma_k(j)>\sigma_k(a_k)$ for all $j>b_k$ and
$\sigma_k(j)<\sigma_k(a_k)$ for all $j<a_k$, which explains the last part
(these two conditions are equivalent). The other conditions are easy to check.
\end{proof}

Notice that one has to carefully choose the obstructions in the implementation
because some of them do not spare enough time to compensate the time they consume.
For example, the following are apparently not good enough
(we therefore omit the proof):

\begin{lemma}\label{simpobstr_7}
In Algorithm \ref{EnumerateWirings2}, arriving at junction $k$ the enumeration will never
yield a simplicial wiring if one of the following conditions is satisfied:
\item \begin{enumerate}
\item $a_k=d_k+1$ and $\sigma_k(d_k+1)=n-1-d_k$,
\item $u_k>1$ and $\sigma_k(u_k)=n+2-u_k$ and $\sigma_k(u_k-1)=n+3-u_k$,
\item $d_k\le n-1$ and $\sigma_k(d_k)=n-d_k$ and $\sigma_k(d_k+1)=n-1-d_k$,
\item $\sigma_k(1)=2$ and $b_k+2<n$ and $\sigma_k(b_k)=3$
and $\sigma_k(b_k+1)=4$ and $\sigma_k(b_k+2)=n$.
\end{enumerate}
\end{lemma}

We now give the algorithm for the simplicial case.
It proceeds in two steps:
\begin{itemize}
\item Compute a list of beginnings as described in Lemma \ref{wirbeg} and
choose best representatives as proposed in Lemma \ref{wirbeg_dihedral} and
Rem.\ \ref{goodchoice}.
\item For each representative of beginnings, create a wiring fragment $\omega$.
Call ``{\bf EnumerateSimplicialWirings} $(\omega,\ell)$'' defined below,
where $\ell$ is the maximum of the $\ell(v)$ for $v$ a vertex in $\omega$ (see
Rem.\ \ref{maxjunction}). Notice that this step can easily be parallelized.
\end{itemize}

\algo{EnumerateSimplicialWirings}{$\omega$,$\ell$}
{Enumerates simplicial wirings starting by $\omega$ with maximal $\ell(v)=\ell$,
at least one from each isomorphism class}
{A wiring fragment $\omega$, $\ell\in\NN$}
{List of completions of the fragment $\omega$}
{
\item Compute the numbers $d_m$ and $u_m$ for $\omega$.
\item If $d_m=u_m$, then if $v_m^i+\varepsilon_m^{n-i}=3$ for all $i=0,\ldots,n$
return $\{\omega\}$, else return $\emptyset$.
\item Check the obstructions of Lemma \ref{simpobstr_1}, \ref{simpobstr_2}, \ref{simpobstr_5},
\ref{simpobstr_4} and return $\emptyset$ if one of them is satisfied.
\item $R\leftarrow \emptyset$.
\item If $b_m\le u_m$ and $v_m^{b_m-1}\le 1$ then find the greatest $i$ with
$\sigma_m(b_m)<\sigma_m(b_m+1)<\ldots<\sigma_m(i)$ (see Lemma \ref{all2}).\\
If $i-b_m<\ell$,
update all the data to a new fragment $\omega'$ with $\Spairs'(m+1)=(b_m,i)$,
and call
\[ R:=R\:\cup\:\text{\bf EnumerateSimplicialWirings}(\omega',\ell).\]
Use $s_i$ to ensure that Lemma \ref{twosquares} is satisfied.
\item If $\sigma_m(b_m)=n-u_m$ or $d_m\ge a_m$ then return $R$.
\item For all $d_m\le i<j\le a_m$ with $j-i<\ell$ do:
\item If $\sigma_m(i)<\sigma_m(i+1)<\ldots<\sigma_m(j-1)<\sigma_m(j)$,
$v_m^{i-1}=1$ and $v_m^{j}=1$ (see Lemma \ref{all2}), then
update all the data to a new fragment $\omega'$ with $\Spairs'(m+1)=(i,j)$,
and call
\[ R:=R\:\cup\:\text{\bf EnumerateSimplicialWirings}(\omega',\ell).\]
\item Return $R$.
}

When the enumeration is complete, one still has to collect the
wirings up to isomorphisms. We use the following observation:
\begin{lemma}
Let $\Ac$ and $\Ac'$ be simplicial arrangements. Then
$\Ac$ and $\Ac'$ are isomorphic if and only if the graphs given by
the corresponding triangulations are isomorphic
(we do not need to require a bijection between the $2$-cells preserving the incidence).
\end{lemma}
\begin{proof}
It suffices to prove that for vertices $v_1$, $v_2$, $v_3$ such that
$(v_1,v_2)$, $(v_2,v_3)$, $(v_1,v_3)$ are edges, the triple $(v_1,v_2,v_3)$
is always a $2$-cell. But a pseudoline crossing $(v_1,v_2,v_3)$ would have to
go through two of the three vertices which is impossible.
\end{proof}

In other words, we just need to test whether certain graphs are isomorphic.
Such a test is implemented in most computer algebra systems that
include combinatorics and is good enough for our purpose.

\section{Stretchable simplicial arrangements}

We now assume that we have a complete list of simplicial wirings of
$n$ lines. A very valuable necessary condition for stretchability is Pappus' Theorem:
\begin{theor}[Pappus]
Let $x,y,z,u,v,w\in\RR^3$ with
\[ \dim\langle x,y,z\rangle = 2 = \dim\langle u,v,w\rangle. \]
Then
\[\dim(
(\langle x,v\rangle\cap\langle y,u\rangle) +
(\langle x,w\rangle\cap\langle z,u\rangle) +
(\langle z,v\rangle\cap\langle y,w\rangle)
) = 2.\]
\end{theor}
We use this theorem in the following way for a wiring $\omega$
(compare \cite[Thm.\ 3.1]{b-G-72}):
Assume that $x,y,z$ are distinct vertices
on one line and $u,v,w$ are distinct vertices on another line.
If there are lines
$\langle x,v\rangle$, $\langle y,u\rangle$, $\langle x,w\rangle$,
$\langle z,u\rangle$, $\langle z,v\rangle$, $\langle y,w\rangle$
in $\omega$, and exactly two of the intersection points
\[ \langle x,v\rangle\cap\langle y,u\rangle,\quad
\langle x,w\rangle\cap\langle z,u\rangle,\quad
\langle z,v\rangle\cap\langle y,w\rangle \]
lie on a line of $\omega$, then stretchability of $\omega$ would
contradict Pappus' Theorem.

This is a very expensive test in terms of running time when implemented.
Therefore it is not possible to include it into the above algorithm.
However, it turns out that almost all unstretchable simplicial wirings
of up to $27$ lines do not satisfy Pappus' Theorem. Thus it appears to
be the best (known) tool to rule out wirings a posteriori.

For the very few remaining simplicial wirings we use \cite[Alg.\ 4.4]{p-C10b}:
We compute a Gr\"obner basis of the ideal given by the incidence
constraints. This either yields a realization of the wiring as arrangement
of straight lines, or it proves that the wiring is unstretchable.

\section{Results}

\begin{figure}
\begin{tiny}
\begin{tabular}{|l|rrrrrrrrrrrrr|}
\hline
number of lines & 2 & 3 & 4 & 5 & 6 & 7 & 8 & 9 & 10 & 11 & 12 & 13 & 14 \\
\hline
stretchable & 0 & 0 & 0 & 0 & 1 & 1 & 1 & 1 & 3 & 1 & 3 & 4 & 4 \\
unstretchable & 0 & 0 & 0 & 0 & 0 & 0 & 0 & 0 & 0 & 0 & 0 & 0 & 0 \\
\hline
number of lines & 15 & 16 & 17 & 18 & 19 & 20 & 21 & 22 & 23 & 24 & 25 & 26 & 27\\
\hline
stretchable & 5 & 7 & 8 & 8 & 7 & 5 & 7 & 5 & 2 & 4 & 8 & 4 & 4 \\
unstr.\ by Pappus & 2 & 6 & 7 & 28 & 35 & 136 & 168 & 978 & 1276 & 12720 & ? & ? & ? \\
unstr.\ but Pappus & 0 & 1 & 1 & 1 & 1 & 0 & 0 & 1 & 0 & 0 & 0 & 1 & 0 \\
\hline
\end{tabular}
\end{tiny}
\caption{Numbers of isomorphism classes of simplicial arrangements
(without the near-pencil arrangements)\label{pseudoiso}}
\end{figure}

\begin{figure}
\vspace{20pt}
\begin{center}
\setlength{\unitlength}{0.8pt}
\begin{picture}(440,90)(40,280)
\moveto(58,305)\lineto(58,305)\lineto(66,325)\lineto(66,325)\lineto(74,335)\lineto(74,335)\lineto(82,340)\lineto(82,340)\lineto(90,350)\lineto(90,350)\lineto(98,355)\lineto(98,355)\lineto(106,375)\lineto(106,375)\lineto(114,380)\lineto(362,380)
\moveto(58,310)\lineto(58,310)\lineto(66,320)\lineto(138,320)\lineto(146,330)\lineto(146,330)\lineto(154,340)\lineto(154,340)\lineto(162,345)\lineto(162,345)\lineto(170,355)\lineto(170,355)\lineto(178,365)\lineto(178,365)\lineto(186,375)\lineto(362,375)
\moveto(58,315)\lineto(58,315)\lineto(66,315)\lineto(186,315)\lineto(194,320)\lineto(194,320)\lineto(202,330)\lineto(202,330)\lineto(210,345)\lineto(210,345)\lineto(218,355)\lineto(218,355)\lineto(226,365)\lineto(226,365)\lineto(234,370)\lineto(362,370)
\moveto(58,320)\lineto(58,320)\lineto(66,310)\lineto(234,310)\lineto(242,320)\lineto(242,320)\lineto(250,330)\lineto(250,330)\lineto(258,335)\lineto(258,335)\lineto(266,345)\lineto(266,345)\lineto(274,355)\lineto(274,355)\lineto(282,365)\lineto(362,365)
\moveto(58,325)\lineto(58,325)\lineto(66,305)\lineto(306,305)\lineto(314,315)\lineto(314,315)\lineto(322,320)\lineto(322,320)\lineto(330,330)\lineto(330,330)\lineto(338,335)\lineto(338,335)\lineto(346,355)\lineto(346,355)\lineto(354,360)\lineto(362,360)
\moveto(58,330)\lineto(66,330)\lineto(74,330)\lineto(114,330)\lineto(122,340)\lineto(122,340)\lineto(130,355)\lineto(130,355)\lineto(138,360)\lineto(170,360)\lineto(178,360)\lineto(218,360)\lineto(226,360)\lineto(274,360)\lineto(282,360)\lineto(346,360)\lineto(354,355)\lineto(362,355)
\moveto(58,335)\lineto(66,335)\lineto(74,325)\lineto(138,325)\lineto(146,325)\lineto(194,325)\lineto(202,325)\lineto(242,325)\lineto(250,325)\lineto(290,325)\lineto(298,340)\lineto(338,340)\lineto(346,350)\lineto(362,350)
\moveto(58,340)\lineto(74,340)\lineto(82,335)\lineto(114,335)\lineto(122,335)\lineto(146,335)\lineto(154,335)\lineto(202,335)\lineto(210,340)\lineto(258,340)\lineto(266,340)\lineto(282,340)\lineto(290,345)\lineto(338,345)\lineto(346,345)\lineto(362,345)
\moveto(58,345)\lineto(82,345)\lineto(90,345)\lineto(122,345)\lineto(130,350)\lineto(162,350)\lineto(170,350)\lineto(210,350)\lineto(218,350)\lineto(266,350)\lineto(274,350)\lineto(338,350)\lineto(346,340)\lineto(362,340)
\moveto(58,350)\lineto(82,350)\lineto(90,340)\lineto(114,340)\lineto(122,330)\lineto(138,330)\lineto(146,320)\lineto(186,320)\lineto(194,315)\lineto(234,315)\lineto(242,315)\lineto(298,315)\lineto(306,325)\lineto(322,325)\lineto(330,325)\lineto(354,325)\lineto(362,335)\lineto(362,335)
\moveto(58,355)\lineto(90,355)\lineto(98,350)\lineto(122,350)\lineto(130,345)\lineto(154,345)\lineto(162,340)\lineto(202,340)\lineto(210,335)\lineto(250,335)\lineto(258,330)\lineto(290,330)\lineto(298,335)\lineto(330,335)\lineto(338,330)\lineto(354,330)\lineto(362,330)\lineto(362,330)
\moveto(58,360)\lineto(98,360)\lineto(106,370)\lineto(178,370)\lineto(186,370)\lineto(226,370)\lineto(234,365)\lineto(274,365)\lineto(282,355)\lineto(338,355)\lineto(346,335)\lineto(354,335)\lineto(362,325)\lineto(362,325)
\moveto(58,365)\lineto(98,365)\lineto(106,365)\lineto(170,365)\lineto(178,355)\lineto(210,355)\lineto(218,345)\lineto(258,345)\lineto(266,335)\lineto(290,335)\lineto(298,330)\lineto(322,330)\lineto(330,320)\lineto(362,320)
\moveto(58,370)\lineto(98,370)\lineto(106,360)\lineto(130,360)\lineto(138,355)\lineto(162,355)\lineto(170,345)\lineto(202,345)\lineto(210,330)\lineto(242,330)\lineto(250,320)\lineto(298,320)\lineto(306,320)\lineto(314,320)\lineto(322,315)\lineto(362,315)
\moveto(58,375)\lineto(98,375)\lineto(106,355)\lineto(122,355)\lineto(130,340)\lineto(146,340)\lineto(154,330)\lineto(194,330)\lineto(202,320)\lineto(234,320)\lineto(242,310)\lineto(306,310)\lineto(314,310)\lineto(362,310)
\moveto(58,380)\lineto(106,380)\lineto(114,375)\lineto(178,375)\lineto(186,365)\lineto(218,365)\lineto(226,355)\lineto(266,355)\lineto(274,345)\lineto(282,345)\lineto(290,340)\lineto(290,340)\lineto(298,325)\lineto(298,325)\lineto(306,315)\lineto(306,315)\lineto(314,305)\lineto(362,305)
\strokepath
\end{picture}
\setlength{\unitlength}{0.8pt}
\begin{picture}(440,90)(40,280)
\moveto(58,305)\lineto(58,305)\lineto(66,335)\lineto(66,335)\lineto(74,340)\lineto(74,340)\lineto(82,350)\lineto(82,350)\lineto(90,355)\lineto(90,355)\lineto(98,365)\lineto(98,365)\lineto(106,370)\lineto(106,370)\lineto(114,380)\lineto(114,380)\lineto(122,385)\lineto(394,385)
\moveto(58,310)\lineto(58,310)\lineto(66,330)\lineto(122,330)\lineto(130,340)\lineto(130,340)\lineto(138,355)\lineto(138,355)\lineto(146,370)\lineto(146,370)\lineto(154,380)\lineto(394,380)
\moveto(58,315)\lineto(58,315)\lineto(66,325)\lineto(162,325)\lineto(170,330)\lineto(170,330)\lineto(178,340)\lineto(178,340)\lineto(186,350)\lineto(186,350)\lineto(194,360)\lineto(194,360)\lineto(202,370)\lineto(202,370)\lineto(210,375)\lineto(394,375)
\moveto(58,320)\lineto(58,320)\lineto(66,320)\lineto(210,320)\lineto(218,330)\lineto(218,330)\lineto(226,340)\lineto(226,340)\lineto(234,350)\lineto(234,350)\lineto(242,360)\lineto(242,360)\lineto(250,370)\lineto(394,370)
\moveto(58,325)\lineto(58,325)\lineto(66,315)\lineto(250,315)\lineto(258,320)\lineto(258,320)\lineto(266,330)\lineto(266,330)\lineto(274,340)\lineto(274,340)\lineto(282,350)\lineto(282,350)\lineto(290,360)\lineto(290,360)\lineto(298,365)\lineto(394,365)
\moveto(58,330)\lineto(58,330)\lineto(66,310)\lineto(306,310)\lineto(314,320)\lineto(314,320)\lineto(322,335)\lineto(322,335)\lineto(330,350)\lineto(330,350)\lineto(338,360)\lineto(394,360)
\moveto(58,335)\lineto(58,335)\lineto(66,305)\lineto(362,305)\lineto(370,315)\lineto(370,315)\lineto(378,330)\lineto(378,330)\lineto(386,345)\lineto(386,345)\lineto(394,355)\lineto(394,355)
\moveto(58,340)\lineto(66,340)\lineto(74,335)\lineto(122,335)\lineto(130,335)\lineto(170,335)\lineto(178,335)\lineto(218,335)\lineto(226,335)\lineto(266,335)\lineto(274,335)\lineto(298,335)\lineto(306,340)\lineto(322,340)\lineto(330,345)\lineto(338,345)\lineto(346,350)\lineto(386,350)\lineto(394,350)\lineto(394,350)
\moveto(58,345)\lineto(74,345)\lineto(82,345)\lineto(130,345)\lineto(138,350)\lineto(154,350)\lineto(162,355)\lineto(186,355)\lineto(194,355)\lineto(234,355)\lineto(242,355)\lineto(282,355)\lineto(290,355)\lineto(330,355)\lineto(338,355)\lineto(386,355)\lineto(394,345)\lineto(394,345)
\moveto(58,350)\lineto(74,350)\lineto(82,340)\lineto(122,340)\lineto(130,330)\lineto(162,330)\lineto(170,325)\lineto(210,325)\lineto(218,325)\lineto(258,325)\lineto(266,325)\lineto(314,325)\lineto(322,330)\lineto(346,330)\lineto(354,335)\lineto(378,335)\lineto(386,340)\lineto(394,340)
\moveto(58,355)\lineto(82,355)\lineto(90,350)\lineto(130,350)\lineto(138,345)\lineto(178,345)\lineto(186,345)\lineto(226,345)\lineto(234,345)\lineto(274,345)\lineto(282,345)\lineto(322,345)\lineto(330,340)\lineto(378,340)\lineto(386,335)\lineto(394,335)
\moveto(58,360)\lineto(90,360)\lineto(98,360)\lineto(138,360)\lineto(146,365)\lineto(194,365)\lineto(202,365)\lineto(242,365)\lineto(250,365)\lineto(290,365)\lineto(298,360)\lineto(330,360)\lineto(338,350)\lineto(338,350)\lineto(346,345)\lineto(378,345)\lineto(386,330)\lineto(394,330)
\moveto(58,365)\lineto(90,365)\lineto(98,355)\lineto(130,355)\lineto(138,340)\lineto(170,340)\lineto(178,330)\lineto(210,330)\lineto(218,320)\lineto(250,320)\lineto(258,315)\lineto(306,315)\lineto(314,315)\lineto(354,315)\lineto(362,320)\lineto(370,320)\lineto(378,325)\lineto(394,325)
\moveto(58,370)\lineto(98,370)\lineto(106,365)\lineto(138,365)\lineto(146,360)\lineto(186,360)\lineto(194,350)\lineto(226,350)\lineto(234,340)\lineto(266,340)\lineto(274,330)\lineto(314,330)\lineto(322,325)\lineto(370,325)\lineto(378,320)\lineto(394,320)
\moveto(58,375)\lineto(106,375)\lineto(114,375)\lineto(146,375)\lineto(154,375)\lineto(202,375)\lineto(210,370)\lineto(242,370)\lineto(250,360)\lineto(282,360)\lineto(290,350)\lineto(322,350)\lineto(330,335)\lineto(346,335)\lineto(354,330)\lineto(370,330)\lineto(378,315)\lineto(394,315)
\moveto(58,380)\lineto(106,380)\lineto(114,370)\lineto(138,370)\lineto(146,355)\lineto(154,355)\lineto(162,350)\lineto(178,350)\lineto(186,340)\lineto(218,340)\lineto(226,330)\lineto(258,330)\lineto(266,320)\lineto(306,320)\lineto(314,310)\lineto(362,310)\lineto(370,310)\lineto(394,310)
\moveto(58,385)\lineto(114,385)\lineto(122,380)\lineto(146,380)\lineto(154,370)\lineto(194,370)\lineto(202,360)\lineto(234,360)\lineto(242,350)\lineto(274,350)\lineto(282,340)\lineto(298,340)\lineto(306,335)\lineto(314,335)\lineto(322,320)\lineto(354,320)\lineto(362,315)\lineto(362,315)\lineto(370,305)\lineto(394,305)
\strokepath
\end{picture}
\setlength{\unitlength}{0.8pt}
\begin{picture}(440,90)(40,285)
\moveto(58,305)\lineto(58,305)\lineto(66,340)\lineto(66,340)\lineto(74,345)\lineto(74,345)\lineto(82,355)\lineto(82,355)\lineto(90,360)\lineto(90,360)\lineto(98,370)\lineto(98,370)\lineto(106,375)\lineto(106,375)\lineto(114,385)\lineto(114,385)\lineto(122,390)\lineto(426,390)
\moveto(58,310)\lineto(58,310)\lineto(66,335)\lineto(122,335)\lineto(130,345)\lineto(130,345)\lineto(138,360)\lineto(138,360)\lineto(146,375)\lineto(146,375)\lineto(154,385)\lineto(426,385)
\moveto(58,315)\lineto(58,315)\lineto(66,330)\lineto(154,330)\lineto(162,335)\lineto(162,335)\lineto(170,345)\lineto(170,345)\lineto(178,350)\lineto(178,350)\lineto(186,360)\lineto(186,360)\lineto(194,365)\lineto(194,365)\lineto(202,375)\lineto(202,375)\lineto(210,380)\lineto(426,380)
\moveto(58,320)\lineto(58,320)\lineto(66,325)\lineto(210,325)\lineto(218,335)\lineto(218,335)\lineto(226,350)\lineto(226,350)\lineto(234,365)\lineto(234,365)\lineto(242,375)\lineto(426,375)
\moveto(58,325)\lineto(58,325)\lineto(66,320)\lineto(250,320)\lineto(258,325)\lineto(258,325)\lineto(266,335)\lineto(266,335)\lineto(274,345)\lineto(274,345)\lineto(282,355)\lineto(282,355)\lineto(290,365)\lineto(290,365)\lineto(298,370)\lineto(426,370)
\moveto(58,330)\lineto(58,330)\lineto(66,315)\lineto(298,315)\lineto(306,325)\lineto(306,325)\lineto(314,335)\lineto(314,335)\lineto(322,345)\lineto(322,345)\lineto(330,355)\lineto(330,355)\lineto(338,365)\lineto(426,365)
\moveto(58,335)\lineto(58,335)\lineto(66,310)\lineto(338,310)\lineto(346,315)\lineto(346,315)\lineto(354,325)\lineto(354,325)\lineto(362,335)\lineto(362,335)\lineto(370,345)\lineto(370,345)\lineto(378,355)\lineto(378,355)\lineto(386,360)\lineto(426,360)
\moveto(58,340)\lineto(58,340)\lineto(66,305)\lineto(394,305)\lineto(402,315)\lineto(402,315)\lineto(410,330)\lineto(410,330)\lineto(418,345)\lineto(418,345)\lineto(426,355)\lineto(426,355)
\moveto(58,345)\lineto(66,345)\lineto(74,340)\lineto(122,340)\lineto(130,340)\lineto(162,340)\lineto(170,340)\lineto(218,340)\lineto(226,345)\lineto(242,345)\lineto(250,350)\lineto(274,350)\lineto(282,350)\lineto(322,350)\lineto(330,350)\lineto(370,350)\lineto(378,350)\lineto(418,350)\lineto(426,350)\lineto(426,350)
\moveto(58,350)\lineto(74,350)\lineto(82,350)\lineto(130,350)\lineto(138,355)\lineto(178,355)\lineto(186,355)\lineto(226,355)\lineto(234,360)\lineto(282,360)\lineto(290,360)\lineto(330,360)\lineto(338,360)\lineto(378,360)\lineto(386,355)\lineto(418,355)\lineto(426,345)\lineto(426,345)
\moveto(58,355)\lineto(74,355)\lineto(82,345)\lineto(122,345)\lineto(130,335)\lineto(154,335)\lineto(162,330)\lineto(210,330)\lineto(218,330)\lineto(258,330)\lineto(266,330)\lineto(306,330)\lineto(314,330)\lineto(354,330)\lineto(362,330)\lineto(386,330)\lineto(394,335)\lineto(410,335)\lineto(418,340)\lineto(426,340)
\moveto(58,360)\lineto(82,360)\lineto(90,355)\lineto(130,355)\lineto(138,350)\lineto(170,350)\lineto(178,345)\lineto(218,345)\lineto(226,340)\lineto(266,340)\lineto(274,340)\lineto(314,340)\lineto(322,340)\lineto(362,340)\lineto(370,340)\lineto(410,340)\lineto(418,335)\lineto(426,335)
\moveto(58,365)\lineto(90,365)\lineto(98,365)\lineto(138,365)\lineto(146,370)\lineto(194,370)\lineto(202,370)\lineto(234,370)\lineto(242,370)\lineto(290,370)\lineto(298,365)\lineto(330,365)\lineto(338,355)\lineto(370,355)\lineto(378,345)\lineto(410,345)\lineto(418,330)\lineto(426,330)
\moveto(58,370)\lineto(90,370)\lineto(98,360)\lineto(130,360)\lineto(138,345)\lineto(162,345)\lineto(170,335)\lineto(210,335)\lineto(218,325)\lineto(250,325)\lineto(258,320)\lineto(298,320)\lineto(306,320)\lineto(346,320)\lineto(354,320)\lineto(402,320)\lineto(410,325)\lineto(426,325)
\moveto(58,375)\lineto(98,375)\lineto(106,370)\lineto(138,370)\lineto(146,365)\lineto(186,365)\lineto(194,360)\lineto(226,360)\lineto(234,355)\lineto(274,355)\lineto(282,345)\lineto(314,345)\lineto(322,335)\lineto(354,335)\lineto(362,325)\lineto(402,325)\lineto(410,320)\lineto(426,320)
\moveto(58,380)\lineto(106,380)\lineto(114,380)\lineto(146,380)\lineto(154,380)\lineto(202,380)\lineto(210,375)\lineto(234,375)\lineto(242,365)\lineto(282,365)\lineto(290,355)\lineto(322,355)\lineto(330,345)\lineto(362,345)\lineto(370,335)\lineto(386,335)\lineto(394,330)\lineto(402,330)\lineto(410,315)\lineto(426,315)
\moveto(58,385)\lineto(106,385)\lineto(114,375)\lineto(138,375)\lineto(146,360)\lineto(178,360)\lineto(186,350)\lineto(218,350)\lineto(226,335)\lineto(258,335)\lineto(266,325)\lineto(298,325)\lineto(306,315)\lineto(338,315)\lineto(346,310)\lineto(394,310)\lineto(402,310)\lineto(426,310)
\moveto(58,390)\lineto(114,390)\lineto(122,385)\lineto(146,385)\lineto(154,375)\lineto(194,375)\lineto(202,365)\lineto(226,365)\lineto(234,350)\lineto(242,350)\lineto(250,345)\lineto(266,345)\lineto(274,335)\lineto(306,335)\lineto(314,325)\lineto(346,325)\lineto(354,315)\lineto(394,315)\lineto(402,305)\lineto(426,305)
\strokepath
\end{picture}
\end{center}
\begin{center}
\setlength{\unitlength}{0.6pt}
\begin{picture}(550,90)(50,290)
\moveto(58,305)\lineto(58,305)\lineto(66,320)\lineto(66,320)\lineto(74,330)\lineto(74,330)\lineto(82,340)\lineto(82,340)\lineto(90,345)\lineto(90,345)\lineto(98,355)\lineto(98,355)\lineto(106,365)\lineto(106,365)\lineto(114,375)\lineto(114,375)\lineto(122,380)\lineto(122,380)\lineto(130,390)\lineto(130,390)\lineto(138,395)\lineto(498,395)
\moveto(58,310)\lineto(58,310)\lineto(66,315)\lineto(138,315)\lineto(146,320)\lineto(146,320)\lineto(154,330)\lineto(154,330)\lineto(162,345)\lineto(162,345)\lineto(170,355)\lineto(170,355)\lineto(178,365)\lineto(178,365)\lineto(186,380)\lineto(186,380)\lineto(194,390)\lineto(498,390)
\moveto(58,315)\lineto(58,315)\lineto(66,310)\lineto(298,310)\lineto(306,325)\lineto(306,325)\lineto(314,340)\lineto(314,340)\lineto(322,355)\lineto(322,355)\lineto(330,370)\lineto(330,370)\lineto(338,385)\lineto(498,385)
\moveto(58,320)\lineto(58,320)\lineto(66,305)\lineto(442,305)\lineto(450,315)\lineto(450,315)\lineto(458,330)\lineto(458,330)\lineto(466,340)\lineto(466,340)\lineto(474,350)\lineto(474,350)\lineto(482,365)\lineto(482,365)\lineto(490,375)\lineto(490,375)\lineto(498,380)\lineto(498,380)
\moveto(58,325)\lineto(66,325)\lineto(74,325)\lineto(146,325)\lineto(154,325)\lineto(242,325)\lineto(250,335)\lineto(250,335)\lineto(258,340)\lineto(258,340)\lineto(266,350)\lineto(266,350)\lineto(274,360)\lineto(274,360)\lineto(282,370)\lineto(282,370)\lineto(290,375)\lineto(330,375)\lineto(338,380)\lineto(490,380)\lineto(498,375)\lineto(498,375)
\moveto(58,330)\lineto(66,330)\lineto(74,320)\lineto(138,320)\lineto(146,315)\lineto(298,315)\lineto(306,320)\lineto(346,320)\lineto(354,325)\lineto(354,325)\lineto(362,335)\lineto(362,335)\lineto(370,345)\lineto(370,345)\lineto(378,355)\lineto(378,355)\lineto(386,360)\lineto(386,360)\lineto(394,370)\lineto(482,370)\lineto(490,370)\lineto(498,370)
\moveto(58,335)\lineto(74,335)\lineto(82,335)\lineto(154,335)\lineto(162,340)\lineto(202,340)\lineto(210,345)\lineto(210,345)\lineto(218,360)\lineto(218,360)\lineto(226,370)\lineto(226,370)\lineto(234,380)\lineto(330,380)\lineto(338,375)\lineto(482,375)\lineto(490,365)\lineto(498,365)
\moveto(58,340)\lineto(74,340)\lineto(82,330)\lineto(146,330)\lineto(154,320)\lineto(298,320)\lineto(306,315)\lineto(402,315)\lineto(410,325)\lineto(410,325)\lineto(418,335)\lineto(418,335)\lineto(426,350)\lineto(426,350)\lineto(434,355)\lineto(474,355)\lineto(482,360)\lineto(498,360)
\moveto(58,345)\lineto(82,345)\lineto(90,340)\lineto(154,340)\lineto(162,335)\lineto(234,335)\lineto(242,345)\lineto(258,345)\lineto(266,345)\lineto(314,345)\lineto(322,350)\lineto(370,350)\lineto(378,350)\lineto(394,350)\lineto(402,360)\lineto(474,360)\lineto(482,355)\lineto(498,355)
\moveto(58,350)\lineto(90,350)\lineto(98,350)\lineto(162,350)\lineto(170,350)\lineto(210,350)\lineto(218,355)\lineto(266,355)\lineto(274,355)\lineto(290,355)\lineto(298,360)\lineto(322,360)\lineto(330,365)\lineto(386,365)\lineto(394,365)\lineto(474,365)\lineto(482,350)\lineto(498,350)
\moveto(58,355)\lineto(90,355)\lineto(98,345)\lineto(154,345)\lineto(162,330)\lineto(242,330)\lineto(250,330)\lineto(306,330)\lineto(314,335)\lineto(338,335)\lineto(346,340)\lineto(362,340)\lineto(370,340)\lineto(418,340)\lineto(426,345)\lineto(466,345)\lineto(474,345)\lineto(498,345)
\moveto(58,360)\lineto(98,360)\lineto(106,360)\lineto(170,360)\lineto(178,360)\lineto(194,360)\lineto(202,365)\lineto(218,365)\lineto(226,365)\lineto(274,365)\lineto(282,365)\lineto(322,365)\lineto(330,360)\lineto(378,360)\lineto(386,355)\lineto(394,355)\lineto(402,355)\lineto(426,355)\lineto(434,350)\lineto(466,350)\lineto(474,340)\lineto(498,340)
\moveto(58,365)\lineto(98,365)\lineto(106,355)\lineto(162,355)\lineto(170,345)\lineto(202,345)\lineto(210,340)\lineto(234,340)\lineto(242,340)\lineto(250,340)\lineto(258,335)\lineto(306,335)\lineto(314,330)\lineto(354,330)\lineto(362,330)\lineto(410,330)\lineto(418,330)\lineto(434,330)\lineto(442,335)\lineto(458,335)\lineto(466,335)\lineto(498,335)
\moveto(58,370)\lineto(106,370)\lineto(114,370)\lineto(178,370)\lineto(186,375)\lineto(226,375)\lineto(234,375)\lineto(282,375)\lineto(290,370)\lineto(322,370)\lineto(330,355)\lineto(370,355)\lineto(378,345)\lineto(418,345)\lineto(426,340)\lineto(458,340)\lineto(466,330)\lineto(498,330)
\moveto(58,375)\lineto(106,375)\lineto(114,365)\lineto(170,365)\lineto(178,355)\lineto(210,355)\lineto(218,350)\lineto(258,350)\lineto(266,340)\lineto(306,340)\lineto(314,325)\lineto(346,325)\lineto(354,320)\lineto(402,320)\lineto(410,320)\lineto(450,320)\lineto(458,325)\lineto(498,325)
\moveto(58,380)\lineto(114,380)\lineto(122,375)\lineto(178,375)\lineto(186,370)\lineto(218,370)\lineto(226,360)\lineto(266,360)\lineto(274,350)\lineto(314,350)\lineto(322,345)\lineto(362,345)\lineto(370,335)\lineto(410,335)\lineto(418,325)\lineto(450,325)\lineto(458,320)\lineto(498,320)
\moveto(58,385)\lineto(122,385)\lineto(130,385)\lineto(186,385)\lineto(194,385)\lineto(330,385)\lineto(338,370)\lineto(386,370)\lineto(394,360)\lineto(394,360)\lineto(402,350)\lineto(418,350)\lineto(426,335)\lineto(434,335)\lineto(442,330)\lineto(450,330)\lineto(458,315)\lineto(498,315)
\moveto(58,390)\lineto(122,390)\lineto(130,380)\lineto(178,380)\lineto(186,365)\lineto(194,365)\lineto(202,360)\lineto(210,360)\lineto(218,345)\lineto(234,345)\lineto(242,335)\lineto(242,335)\lineto(250,325)\lineto(298,325)\lineto(306,310)\lineto(442,310)\lineto(450,310)\lineto(498,310)
\moveto(58,395)\lineto(130,395)\lineto(138,390)\lineto(186,390)\lineto(194,380)\lineto(226,380)\lineto(234,370)\lineto(274,370)\lineto(282,360)\lineto(290,360)\lineto(298,355)\lineto(314,355)\lineto(322,340)\lineto(338,340)\lineto(346,335)\lineto(354,335)\lineto(362,325)\lineto(402,325)\lineto(410,315)\lineto(442,315)\lineto(450,305)\lineto(498,305)
\strokepath
\end{picture}
\setlength{\unitlength}{0.5pt}
\begin{picture}(660,90)(50,315)
\moveto(58,305)\lineto(58,305)\lineto(66,350)\lineto(66,350)\lineto(74,355)\lineto(74,355)\lineto(82,365)\lineto(82,365)\lineto(90,375)\lineto(90,375)\lineto(98,385)\lineto(98,385)\lineto(106,395)\lineto(106,395)\lineto(114,405)\lineto(114,405)\lineto(122,410)\lineto(594,410)
\moveto(58,310)\lineto(58,310)\lineto(66,345)\lineto(122,345)\lineto(130,355)\lineto(130,355)\lineto(138,365)\lineto(138,365)\lineto(146,375)\lineto(146,375)\lineto(154,385)\lineto(154,385)\lineto(162,395)\lineto(162,395)\lineto(170,405)\lineto(594,405)
\moveto(58,315)\lineto(58,315)\lineto(66,340)\lineto(178,340)\lineto(186,345)\lineto(186,345)\lineto(194,355)\lineto(194,355)\lineto(202,370)\lineto(202,370)\lineto(210,385)\lineto(210,385)\lineto(218,395)\lineto(218,395)\lineto(226,400)\lineto(594,400)
\moveto(58,320)\lineto(58,320)\lineto(66,335)\lineto(226,335)\lineto(234,345)\lineto(234,345)\lineto(242,355)\lineto(242,355)\lineto(250,360)\lineto(250,360)\lineto(258,370)\lineto(258,370)\lineto(266,375)\lineto(266,375)\lineto(274,385)\lineto(274,385)\lineto(282,395)\lineto(594,395)
\moveto(58,325)\lineto(58,325)\lineto(66,330)\lineto(282,330)\lineto(290,335)\lineto(290,335)\lineto(298,345)\lineto(298,345)\lineto(306,360)\lineto(306,360)\lineto(314,375)\lineto(314,375)\lineto(322,385)\lineto(322,385)\lineto(330,390)\lineto(594,390)
\moveto(58,330)\lineto(58,330)\lineto(66,325)\lineto(330,325)\lineto(338,335)\lineto(338,335)\lineto(346,345)\lineto(346,345)\lineto(354,350)\lineto(354,350)\lineto(362,360)\lineto(362,360)\lineto(370,365)\lineto(370,365)\lineto(378,375)\lineto(378,375)\lineto(386,385)\lineto(594,385)
\moveto(58,335)\lineto(58,335)\lineto(66,320)\lineto(386,320)\lineto(394,325)\lineto(394,325)\lineto(402,335)\lineto(402,335)\lineto(410,350)\lineto(410,350)\lineto(418,365)\lineto(418,365)\lineto(426,375)\lineto(426,375)\lineto(434,380)\lineto(594,380)
\moveto(58,340)\lineto(58,340)\lineto(66,315)\lineto(434,315)\lineto(442,325)\lineto(442,325)\lineto(450,335)\lineto(450,335)\lineto(458,340)\lineto(458,340)\lineto(466,350)\lineto(466,350)\lineto(474,355)\lineto(474,355)\lineto(482,365)\lineto(482,365)\lineto(490,375)\lineto(594,375)
\moveto(58,345)\lineto(58,345)\lineto(66,310)\lineto(490,310)\lineto(498,315)\lineto(498,315)\lineto(506,325)\lineto(506,325)\lineto(514,340)\lineto(514,340)\lineto(522,355)\lineto(522,355)\lineto(530,365)\lineto(530,365)\lineto(538,370)\lineto(594,370)
\moveto(58,350)\lineto(58,350)\lineto(66,305)\lineto(546,305)\lineto(554,315)\lineto(554,315)\lineto(562,325)\lineto(562,325)\lineto(570,335)\lineto(570,335)\lineto(578,345)\lineto(578,345)\lineto(586,355)\lineto(586,355)\lineto(594,365)\lineto(594,365)
\moveto(58,355)\lineto(66,355)\lineto(74,350)\lineto(122,350)\lineto(130,350)\lineto(186,350)\lineto(194,350)\lineto(234,350)\lineto(242,350)\lineto(298,350)\lineto(306,355)\lineto(354,355)\lineto(362,355)\lineto(410,355)\lineto(418,360)\lineto(474,360)\lineto(482,360)\lineto(522,360)\lineto(530,360)\lineto(586,360)\lineto(594,360)\lineto(594,360)
\moveto(58,360)\lineto(74,360)\lineto(82,360)\lineto(130,360)\lineto(138,360)\lineto(194,360)\lineto(202,365)\lineto(250,365)\lineto(258,365)\lineto(306,365)\lineto(314,370)\lineto(370,370)\lineto(378,370)\lineto(418,370)\lineto(426,370)\lineto(482,370)\lineto(490,370)\lineto(530,370)\lineto(538,365)\lineto(586,365)\lineto(594,355)\lineto(594,355)
\moveto(58,365)\lineto(74,365)\lineto(82,355)\lineto(122,355)\lineto(130,345)\lineto(178,345)\lineto(186,340)\lineto(226,340)\lineto(234,340)\lineto(290,340)\lineto(298,340)\lineto(338,340)\lineto(346,340)\lineto(402,340)\lineto(410,345)\lineto(458,345)\lineto(466,345)\lineto(514,345)\lineto(522,350)\lineto(578,350)\lineto(586,350)\lineto(594,350)
\moveto(58,370)\lineto(82,370)\lineto(90,370)\lineto(138,370)\lineto(146,370)\lineto(170,370)\lineto(178,375)\lineto(202,375)\lineto(210,380)\lineto(266,380)\lineto(274,380)\lineto(314,380)\lineto(322,380)\lineto(378,380)\lineto(386,380)\lineto(426,380)\lineto(434,375)\lineto(482,375)\lineto(490,365)\lineto(522,365)\lineto(530,355)\lineto(578,355)\lineto(586,345)\lineto(594,345)
\moveto(58,375)\lineto(82,375)\lineto(90,365)\lineto(130,365)\lineto(138,355)\lineto(186,355)\lineto(194,345)\lineto(226,345)\lineto(234,335)\lineto(282,335)\lineto(290,330)\lineto(330,330)\lineto(338,330)\lineto(394,330)\lineto(402,330)\lineto(442,330)\lineto(450,330)\lineto(506,330)\lineto(514,335)\lineto(538,335)\lineto(546,340)\lineto(570,340)\lineto(578,340)\lineto(594,340)
\moveto(58,380)\lineto(90,380)\lineto(98,380)\lineto(146,380)\lineto(154,380)\lineto(202,380)\lineto(210,375)\lineto(258,375)\lineto(266,370)\lineto(306,370)\lineto(314,365)\lineto(362,365)\lineto(370,360)\lineto(410,360)\lineto(418,355)\lineto(466,355)\lineto(474,350)\lineto(514,350)\lineto(522,345)\lineto(570,345)\lineto(578,335)\lineto(594,335)
\moveto(58,385)\lineto(90,385)\lineto(98,375)\lineto(138,375)\lineto(146,365)\lineto(194,365)\lineto(202,360)\lineto(242,360)\lineto(250,355)\lineto(298,355)\lineto(306,350)\lineto(346,350)\lineto(354,345)\lineto(402,345)\lineto(410,340)\lineto(450,340)\lineto(458,335)\lineto(506,335)\lineto(514,330)\lineto(562,330)\lineto(570,330)\lineto(594,330)
\moveto(58,390)\lineto(98,390)\lineto(106,390)\lineto(154,390)\lineto(162,390)\lineto(210,390)\lineto(218,390)\lineto(274,390)\lineto(282,390)\lineto(322,390)\lineto(330,385)\lineto(378,385)\lineto(386,375)\lineto(418,375)\lineto(426,365)\lineto(474,365)\lineto(482,355)\lineto(514,355)\lineto(522,340)\lineto(538,340)\lineto(546,335)\lineto(562,335)\lineto(570,325)\lineto(594,325)
\moveto(58,395)\lineto(98,395)\lineto(106,385)\lineto(146,385)\lineto(154,375)\lineto(170,375)\lineto(178,370)\lineto(194,370)\lineto(202,355)\lineto(234,355)\lineto(242,345)\lineto(290,345)\lineto(298,335)\lineto(330,335)\lineto(338,325)\lineto(386,325)\lineto(394,320)\lineto(434,320)\lineto(442,320)\lineto(498,320)\lineto(506,320)\lineto(554,320)\lineto(562,320)\lineto(594,320)
\moveto(58,400)\lineto(106,400)\lineto(114,400)\lineto(162,400)\lineto(170,400)\lineto(218,400)\lineto(226,395)\lineto(274,395)\lineto(282,385)\lineto(314,385)\lineto(322,375)\lineto(370,375)\lineto(378,365)\lineto(410,365)\lineto(418,350)\lineto(458,350)\lineto(466,340)\lineto(506,340)\lineto(514,325)\lineto(554,325)\lineto(562,315)\lineto(594,315)
\moveto(58,405)\lineto(106,405)\lineto(114,395)\lineto(154,395)\lineto(162,385)\lineto(202,385)\lineto(210,370)\lineto(250,370)\lineto(258,360)\lineto(298,360)\lineto(306,345)\lineto(338,345)\lineto(346,335)\lineto(394,335)\lineto(402,325)\lineto(434,325)\lineto(442,315)\lineto(490,315)\lineto(498,310)\lineto(546,310)\lineto(554,310)\lineto(594,310)
\moveto(58,410)\lineto(114,410)\lineto(122,405)\lineto(162,405)\lineto(170,395)\lineto(210,395)\lineto(218,385)\lineto(266,385)\lineto(274,375)\lineto(306,375)\lineto(314,360)\lineto(354,360)\lineto(362,350)\lineto(402,350)\lineto(410,335)\lineto(442,335)\lineto(450,325)\lineto(498,325)\lineto(506,315)\lineto(546,315)\lineto(554,305)\lineto(594,305)
\strokepath
\end{picture}
\setlength{\unitlength}{0.4pt}
\begin{picture}(760,140)(80,320)
\moveto(58,305)\lineto(58,305)\lineto(66,325)\lineto(66,325)\lineto(74,330)\lineto(74,330)\lineto(82,340)\lineto(82,340)\lineto(90,350)\lineto(90,350)\lineto(98,365)\lineto(98,365)\lineto(106,375)\lineto(106,375)\lineto(114,385)\lineto(114,385)\lineto(122,390)\lineto(122,390)\lineto(130,410)\lineto(130,410)\lineto(138,415)\lineto(138,415)\lineto(146,425)\lineto(146,425)\lineto(154,430)\lineto(866,430)
\moveto(58,310)\lineto(58,310)\lineto(66,320)\lineto(178,320)\lineto(186,330)\lineto(186,330)\lineto(194,340)\lineto(194,340)\lineto(202,350)\lineto(202,350)\lineto(210,360)\lineto(210,360)\lineto(218,365)\lineto(218,365)\lineto(226,375)\lineto(226,375)\lineto(234,385)\lineto(234,385)\lineto(242,395)\lineto(242,395)\lineto(250,400)\lineto(250,400)\lineto(258,415)\lineto(258,415)\lineto(266,425)\lineto(866,425)
\moveto(58,315)\lineto(58,315)\lineto(66,315)\lineto(434,315)\lineto(442,325)\lineto(442,325)\lineto(450,340)\lineto(450,340)\lineto(458,345)\lineto(458,345)\lineto(466,355)\lineto(466,355)\lineto(474,365)\lineto(474,365)\lineto(482,375)\lineto(482,375)\lineto(490,380)\lineto(490,380)\lineto(498,390)\lineto(498,390)\lineto(506,400)\lineto(506,400)\lineto(514,410)\lineto(514,410)\lineto(522,420)\lineto(866,420)
\moveto(58,320)\lineto(58,320)\lineto(66,310)\lineto(546,310)\lineto(554,315)\lineto(554,315)\lineto(562,325)\lineto(562,325)\lineto(570,330)\lineto(570,330)\lineto(578,350)\lineto(578,350)\lineto(586,355)\lineto(586,355)\lineto(594,365)\lineto(594,365)\lineto(602,375)\lineto(602,375)\lineto(610,390)\lineto(610,390)\lineto(618,400)\lineto(618,400)\lineto(626,410)\lineto(626,410)\lineto(634,415)\lineto(866,415)
\moveto(58,325)\lineto(58,325)\lineto(66,305)\lineto(714,305)\lineto(722,320)\lineto(722,320)\lineto(730,330)\lineto(730,330)\lineto(738,340)\lineto(738,340)\lineto(746,350)\lineto(746,350)\lineto(754,360)\lineto(754,360)\lineto(762,380)\lineto(762,380)\lineto(770,385)\lineto(770,385)\lineto(778,400)\lineto(778,400)\lineto(786,410)\lineto(866,410)
\moveto(58,330)\lineto(66,330)\lineto(74,325)\lineto(178,325)\lineto(186,325)\lineto(362,325)\lineto(370,335)\lineto(370,335)\lineto(378,345)\lineto(378,345)\lineto(386,350)\lineto(386,350)\lineto(394,365)\lineto(394,365)\lineto(402,375)\lineto(402,375)\lineto(410,395)\lineto(410,395)\lineto(418,405)\lineto(506,405)\lineto(514,405)\lineto(618,405)\lineto(626,405)\lineto(778,405)\lineto(786,405)\lineto(866,405)
\moveto(58,335)\lineto(74,335)\lineto(82,335)\lineto(186,335)\lineto(194,335)\lineto(282,335)\lineto(290,345)\lineto(290,345)\lineto(298,365)\lineto(298,365)\lineto(306,375)\lineto(306,375)\lineto(314,390)\lineto(314,390)\lineto(322,395)\lineto(322,395)\lineto(330,405)\lineto(330,405)\lineto(338,415)\lineto(514,415)\lineto(522,415)\lineto(626,415)\lineto(634,410)\lineto(778,410)\lineto(786,400)\lineto(866,400)
\moveto(58,340)\lineto(74,340)\lineto(82,330)\lineto(178,330)\lineto(186,320)\lineto(434,320)\lineto(442,320)\lineto(554,320)\lineto(562,320)\lineto(650,320)\lineto(658,330)\lineto(658,330)\lineto(666,340)\lineto(666,340)\lineto(674,350)\lineto(674,350)\lineto(682,365)\lineto(682,365)\lineto(690,375)\lineto(690,375)\lineto(698,380)\lineto(698,380)\lineto(706,390)\lineto(770,390)\lineto(778,395)\lineto(866,395)
\moveto(58,345)\lineto(82,345)\lineto(90,345)\lineto(194,345)\lineto(202,345)\lineto(274,345)\lineto(282,350)\lineto(290,350)\lineto(298,360)\lineto(338,360)\lineto(346,365)\lineto(346,365)\lineto(354,375)\lineto(354,375)\lineto(362,380)\lineto(402,380)\lineto(410,390)\lineto(418,390)\lineto(426,395)\lineto(498,395)\lineto(506,395)\lineto(610,395)\lineto(618,395)\lineto(770,395)\lineto(778,390)\lineto(866,390)
\moveto(58,350)\lineto(82,350)\lineto(90,340)\lineto(186,340)\lineto(194,330)\lineto(362,330)\lineto(370,330)\lineto(442,330)\lineto(450,335)\lineto(570,335)\lineto(578,345)\lineto(634,345)\lineto(642,355)\lineto(674,355)\lineto(682,360)\lineto(706,360)\lineto(714,365)\lineto(754,365)\lineto(762,375)\lineto(786,375)\lineto(794,385)\lineto(866,385)
\moveto(58,355)\lineto(90,355)\lineto(98,360)\lineto(170,360)\lineto(178,370)\lineto(218,370)\lineto(226,370)\lineto(298,370)\lineto(306,370)\lineto(346,370)\lineto(354,370)\lineto(394,370)\lineto(402,370)\lineto(474,370)\lineto(482,370)\lineto(522,370)\lineto(530,380)\lineto(602,380)\lineto(610,385)\lineto(698,385)\lineto(706,385)\lineto(762,385)\lineto(770,380)\lineto(786,380)\lineto(794,380)\lineto(866,380)
\moveto(58,360)\lineto(90,360)\lineto(98,355)\lineto(202,355)\lineto(210,355)\lineto(290,355)\lineto(298,355)\lineto(386,355)\lineto(394,360)\lineto(466,360)\lineto(474,360)\lineto(530,360)\lineto(538,370)\lineto(594,370)\lineto(602,370)\lineto(682,370)\lineto(690,370)\lineto(754,370)\lineto(762,370)\lineto(794,370)\lineto(802,375)\lineto(866,375)
\moveto(58,365)\lineto(90,365)\lineto(98,350)\lineto(194,350)\lineto(202,340)\lineto(282,340)\lineto(290,340)\lineto(370,340)\lineto(378,340)\lineto(426,340)\lineto(434,350)\lineto(458,350)\lineto(466,350)\lineto(538,350)\lineto(546,360)\lineto(586,360)\lineto(594,360)\lineto(674,360)\lineto(682,355)\lineto(746,355)\lineto(754,355)\lineto(802,355)\lineto(810,360)\lineto(810,360)\lineto(818,370)\lineto(866,370)
\moveto(58,370)\lineto(98,370)\lineto(106,370)\lineto(162,370)\lineto(170,380)\lineto(226,380)\lineto(234,380)\lineto(306,380)\lineto(314,385)\lineto(402,385)\lineto(410,385)\lineto(490,385)\lineto(498,385)\lineto(602,385)\lineto(610,380)\lineto(690,380)\lineto(698,375)\lineto(754,375)\lineto(762,365)\lineto(810,365)\lineto(818,365)\lineto(866,365)
\moveto(58,375)\lineto(98,375)\lineto(106,365)\lineto(170,365)\lineto(178,365)\lineto(210,365)\lineto(218,360)\lineto(290,360)\lineto(298,350)\lineto(378,350)\lineto(386,345)\lineto(426,345)\lineto(434,345)\lineto(450,345)\lineto(458,340)\lineto(570,340)\lineto(578,340)\lineto(642,340)\lineto(650,345)\lineto(666,345)\lineto(674,345)\lineto(738,345)\lineto(746,345)\lineto(818,345)\lineto(826,360)\lineto(866,360)
\moveto(58,380)\lineto(106,380)\lineto(114,380)\lineto(154,380)\lineto(162,390)\lineto(234,390)\lineto(242,390)\lineto(266,390)\lineto(274,400)\lineto(322,400)\lineto(330,400)\lineto(410,400)\lineto(418,400)\lineto(498,400)\lineto(506,390)\lineto(602,390)\lineto(610,375)\lineto(682,375)\lineto(690,365)\lineto(706,365)\lineto(714,360)\lineto(746,360)\lineto(754,350)\lineto(818,350)\lineto(826,355)\lineto(866,355)
\moveto(58,385)\lineto(106,385)\lineto(114,375)\lineto(162,375)\lineto(170,375)\lineto(218,375)\lineto(226,365)\lineto(290,365)\lineto(298,345)\lineto(370,345)\lineto(378,335)\lineto(442,335)\lineto(450,330)\lineto(562,330)\lineto(570,325)\lineto(650,325)\lineto(658,325)\lineto(722,325)\lineto(730,325)\lineto(834,325)\lineto(842,335)\lineto(842,335)\lineto(850,340)\lineto(850,340)\lineto(858,350)\lineto(866,350)
\moveto(58,390)\lineto(114,390)\lineto(122,385)\lineto(154,385)\lineto(162,385)\lineto(226,385)\lineto(234,375)\lineto(298,375)\lineto(306,365)\lineto(338,365)\lineto(346,360)\lineto(386,360)\lineto(394,355)\lineto(458,355)\lineto(466,345)\lineto(570,345)\lineto(578,335)\lineto(658,335)\lineto(666,335)\lineto(730,335)\lineto(738,335)\lineto(826,335)\lineto(834,345)\lineto(850,345)\lineto(858,345)\lineto(866,345)
\moveto(58,395)\lineto(122,395)\lineto(130,405)\lineto(250,405)\lineto(258,410)\lineto(330,410)\lineto(338,410)\lineto(506,410)\lineto(514,400)\lineto(610,400)\lineto(618,390)\lineto(698,390)\lineto(706,380)\lineto(754,380)\lineto(762,360)\lineto(802,360)\lineto(810,355)\lineto(818,355)\lineto(826,350)\lineto(850,350)\lineto(858,340)\lineto(866,340)
\moveto(58,400)\lineto(122,400)\lineto(130,400)\lineto(242,400)\lineto(250,395)\lineto(266,395)\lineto(274,395)\lineto(314,395)\lineto(322,390)\lineto(402,390)\lineto(410,380)\lineto(482,380)\lineto(490,375)\lineto(522,375)\lineto(530,375)\lineto(594,375)\lineto(602,365)\lineto(674,365)\lineto(682,350)\lineto(738,350)\lineto(746,340)\lineto(826,340)\lineto(834,340)\lineto(842,340)\lineto(850,335)\lineto(866,335)
\moveto(58,405)\lineto(122,405)\lineto(130,395)\lineto(234,395)\lineto(242,385)\lineto(306,385)\lineto(314,380)\lineto(354,380)\lineto(362,375)\lineto(394,375)\lineto(402,365)\lineto(466,365)\lineto(474,355)\lineto(538,355)\lineto(546,355)\lineto(578,355)\lineto(586,350)\lineto(634,350)\lineto(642,350)\lineto(666,350)\lineto(674,340)\lineto(730,340)\lineto(738,330)\lineto(834,330)\lineto(842,330)\lineto(866,330)
\moveto(58,410)\lineto(122,410)\lineto(130,390)\lineto(154,390)\lineto(162,380)\lineto(162,380)\lineto(170,370)\lineto(170,370)\lineto(178,360)\lineto(202,360)\lineto(210,350)\lineto(274,350)\lineto(282,345)\lineto(282,345)\lineto(290,335)\lineto(362,335)\lineto(370,325)\lineto(434,325)\lineto(442,315)\lineto(546,315)\lineto(554,310)\lineto(714,310)\lineto(722,315)\lineto(858,315)\lineto(866,325)\lineto(866,325)
\moveto(58,415)\lineto(130,415)\lineto(138,410)\lineto(250,410)\lineto(258,405)\lineto(322,405)\lineto(330,395)\lineto(402,395)\lineto(410,375)\lineto(474,375)\lineto(482,365)\lineto(530,365)\lineto(538,365)\lineto(586,365)\lineto(594,355)\lineto(634,355)\lineto(642,345)\lineto(642,345)\lineto(650,340)\lineto(658,340)\lineto(666,330)\lineto(722,330)\lineto(730,320)\lineto(858,320)\lineto(866,320)\lineto(866,320)
\moveto(58,420)\lineto(138,420)\lineto(146,420)\lineto(258,420)\lineto(266,420)\lineto(514,420)\lineto(522,410)\lineto(618,410)\lineto(626,400)\lineto(770,400)\lineto(778,385)\lineto(786,385)\lineto(794,375)\lineto(794,375)\lineto(802,370)\lineto(810,370)\lineto(818,360)\lineto(818,360)\lineto(826,345)\lineto(826,345)\lineto(834,335)\lineto(834,335)\lineto(842,325)\lineto(858,325)\lineto(866,315)\lineto(866,315)
\moveto(58,425)\lineto(138,425)\lineto(146,415)\lineto(250,415)\lineto(258,400)\lineto(266,400)\lineto(274,390)\lineto(306,390)\lineto(314,375)\lineto(346,375)\lineto(354,365)\lineto(386,365)\lineto(394,350)\lineto(426,350)\lineto(434,340)\lineto(442,340)\lineto(450,325)\lineto(554,325)\lineto(562,315)\lineto(714,315)\lineto(722,310)\lineto(866,310)
\moveto(58,430)\lineto(146,430)\lineto(154,425)\lineto(258,425)\lineto(266,415)\lineto(330,415)\lineto(338,405)\lineto(410,405)\lineto(418,395)\lineto(418,395)\lineto(426,390)\lineto(490,390)\lineto(498,380)\lineto(522,380)\lineto(530,370)\lineto(530,370)\lineto(538,360)\lineto(538,360)\lineto(546,350)\lineto(570,350)\lineto(578,330)\lineto(650,330)\lineto(658,320)\lineto(714,320)\lineto(722,305)\lineto(866,305)
\strokepath
\end{picture}
\end{center}
\caption{The unstretchable simplicial wirings with $16$, $17$, $18$, $19$, $22$, and $26$ lines
that satisfy Pappus'  theorem.\label{nonstwirpap}}
\end{figure}
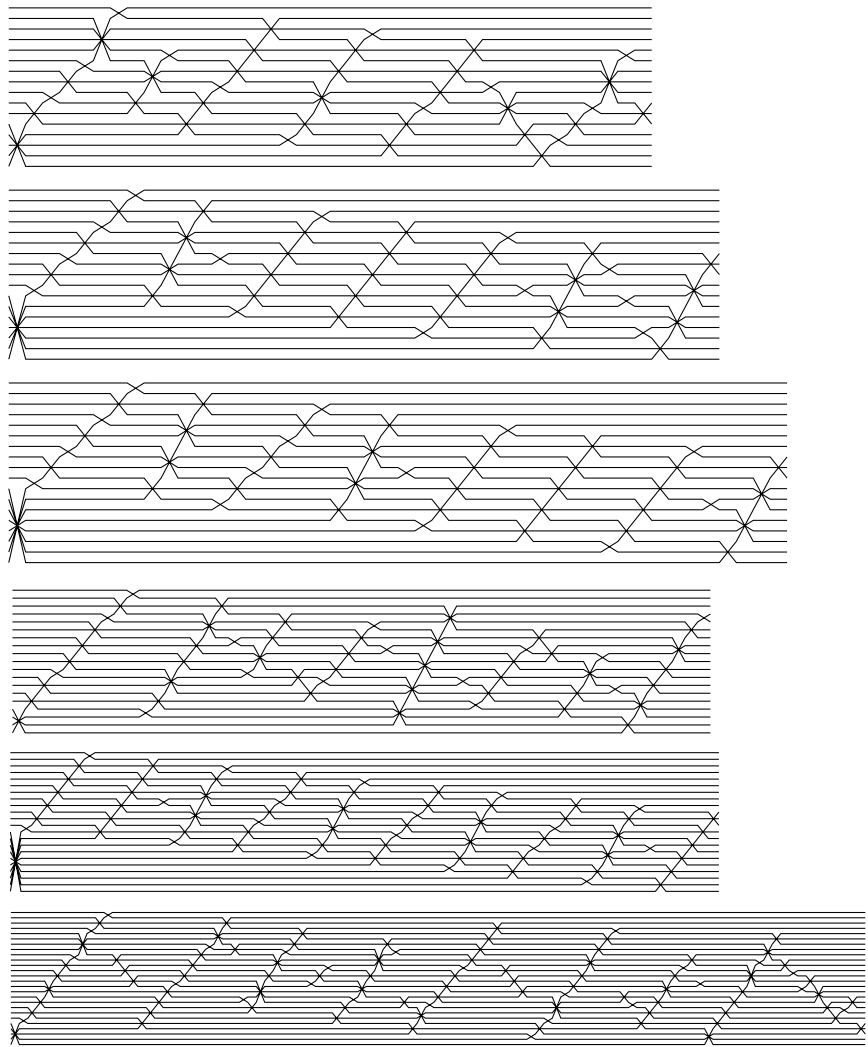

We now summarize the output of the computation.
The running time is exponential in the number of lines:
It takes only $4$ seconds to find all simplicial arrangements of pseudolines
with $20$ lines, but for $26$ lines an ordinary PC needs about two weeks.
Note that this is an implementation in {\sc C} and {\sc C++}; the resulting
data is then processed by several functions implemented in {\sc Magma}.
Notice also that using Lemma \ref{wirbeg}, it is easy to parallelize the
algorithm.
This way we also classify the simplicial arrangements with $27$ lines in
approximately one week, but to reach an enumeration of all simplicial arrangements
with $28$ lines one would require some new idea (or even more computational resources).
But just one step further would not be a great benefit.

Fig.\ \ref{pseudoiso} shows the numbers of simplicial arrangements of pseudolines
up to isomorphisms.
The entries marked ``?'' could, in principle, be computed easily but would
require a better implementation of the function collecting the
wirings up to isomorphisms or possibly just more computational resources.

Fig.\ \ref{hassediag} is a Hasse diagram of the stretchable
simplicial arrangements with up to $27$ lines.
An edge between two arrangements means that one can remove lines from the
larger one to obtain the smaller one. We have included this diagram here
because it appears to be slightly different from \cite[Fig.\ 4]{p-G-09}.
It is less well-arranged than \cite[Fig.\ 4]{p-G-09} because it includes
the connections to the infinite series $\Ac(n,1)$ and some more
edges for instance to $\Ac(7,1)$.
It appears that $\Ac(21,4)$ is not maximal; it is contained in the arrangement
$\Ac(26,4)$.

The following table is a list of the invariants of all (stretchable) simplicial
arrangements with up to $27$ lines except the near-pencil (compare \cite[pp.\ 6--10]{p-G-09}).
The numbers $f_0,f_1,f_2$ are the numbers of vertices, edges and $2$-cells;
$t_i$ is the number of vertices which lie on exactly $i$ lines;
$r_i$ is the number of lines on which exactly $i$ vertices lie.
The last column shows the automorphism groups of the cell complexes given by
the arrangements. The symbols $A_n$, $B_n$, $G_2$, $H_3$ are reflection groups of the
corresponding type, $\Dc_n$ is the dihedral group of order $2n$,
$\Alt_n$ is the alternating group, and $Z_4$ is the cyclic group of order $4$.

Notice that the pairs $(\Ac(17,2),\Ac(17,4))$, $(\Ac(18,4),\Ac(18,5))$, and $(\Ac(19,4),\Ac(19,5))$ are
not uniquely determined by their invariants.
We distinguish $\Ac(17,2)$ and $\Ac(17,4)$ by their position in the Hasse diagram.
$\Ac(19,4)$ and $\Ac(19,5))$ are not distinguishable in \cite[Fig.\ 4]{p-G-09}; however
they have different automorphism groups. The realization of $\Ac(19,5)$ in \cite{p-G-09}
reveals the existence of an automorphism of order $4$ which does not exist
for $\Ac(19,4)$.
The arrangements $\Ac(18,4)$ and $\Ac(18,5)$ are not distinguishable by their
invariants, they even have isomorphic automorphism groups.
But since their role in both Hasse diagrams is the same, the tables in \cite{p-G-09}
and the present table are consistent.

\begin{tiny}
\begin{longtable}{l|llll}
label & $(f_0,f_1,f_2)$ & $(t_2,t_3,\ldots)$ & $(r_2,r_3,\ldots)$ & $\Aut(\Ac)$ \\
\hline
(6,1) & (7,18,12) & (3,4) & (0,6) & $B_3$ \\
\hline
(7,1) & (9,24,16) & (3,6) & (0,4,3) & $B_3$ \\
\hline
(8,1) & (11,30,20) & (4,6,1) & (0,2,6) & $A_1\times B_2$ \\
\hline
(9,1) & (13,36,24) & (6,4,3) & (0,0,9) & $B_3$ \\
\hline
(10,1) & (16,45,30) & (5,10,0,1) & (0,0,5,5) & $\Dc_{10}$ \\
(10,2) & (16,45,30) & (6,7,3) & (0,0,6,3,1) & $A_1\times A_2$ \\
(10,3) & (16,45,30) & (6,7,3) & (0,1,3,6) & $A_1\times G_2$ \\
\hline
(11,1) & (19,54,36) & (7,8,4) & (0,0,4,4,3) & $A_1\times A_1\times A_1$ \\
\hline
(12,1) & (22,63,42) & (6,15,0,0,1) & (0,0,3,3,6) & $A_1\times G_2$ \\
(12,2) & (22,63,42) & (8,10,3,1) & (0,0,3,3,6) & $A_1\times A_1\times A_1$ \\
(12,3) & (22,63,42) & (9,7,6) & (0,0,3,3,6) & $A_1\times A_2$ \\
\hline
(13,1) & (25,72,48) & (9,12,3,0,1) & (0,0,3,0,10) & $A_1\times G_2$ \\
(13,2) & (25,72,48) & (12,4,9) & (0,0,3,0,10) & $B_3$ \\
(13,3) & (25,72,48) & (10,10,3,2) & (0,0,1,4,8) & $A_1\times A_1\times A_1$ \\
(13,4) & (27,78,52) & (6,18,3) & (0,0,0,0,13) & $A_1 \times \Alt_4$ \\
\hline
(14,1) & (29,84,56) & (7,21,0,0,0,1) & (0,0,0,7,0,7) & $\Dc_{14}$ \\
(14,2) & (29,84,56) & (11,12,4,2) & (0,0,1,4,4,4,1) & $A_1\times A_1$ \\
(14,3) & (30,87,58) & (9,16,4,1) & (0,0,0,0,11,3) & $A_1\times A_1$ \\
(14,4) & (29,84,56) & (10,14,4,0,1) & (0,0,0,4,6,4) & $A_1\times A_1\times A_1$ \\
\hline
(15,1) & (31,90,60) & (15,10,0,6) & (0,0,0,0,15) & $H_3$ \\
(15,2) & (33,96,64) & (13,12,6,2) & (0,0,1,4,2,4,4) & $A_1\times B_2$ \\
(15,3) & (34,99,66) & (12,13,9) & (0,0,0,0,9,3,3) & $A_1\times A_2$ \\
(15,4) & (33,96,64) & (12,14,6,0,1) & (0,0,0,0,10,4,1) & $A_1\times A_1\times A_1$ \\
(15,5) & (34,99,66) & (9,22,0,3) & (0,0,0,0,9,3,3) & $A_1\times A_2$ \\
\hline
(16,1) & (37,108,72) & (8,28,0,0,0,0,1) & (0,0,0,4,4,0,8) & $A_1 \times \Dc_8$ \\
(16,2) & (37,108,72) & (14,15,6,1,1) & (0,0,1,2,4,2,7) & $A_1\times A_1\times A_1$ \\
(16,3) & (37,108,72) & (15,13,6,3) & (0,0,0,0,10,0,6) & $A_1\times A_2$ \\
(16,4) & (36,105,70) & (15,15,0,6) & (0,0,0,0,10,5,0,0,1) & $\Dc_{10}$ \\
(16,5) & (37,108,72) & (14,16,3,4) & (0,0,0,2,4,8,0,2) & $A_1\times A_1\times A_1$ \\
(16,6) & (37,108,72) & (15,12,9,0,1) & (0,0,0,0,7,6,3) & $A_1\times A_2$ \\
(16,7) & (38,111,74) & (12,19,6,0,1) & (0,0,0,3,3,2,8) & $A_1\times A_1\times A_1$ \\
\hline
(17,1) & (41,120,80) & (12,24,4,0,0,0,1) & (0,0,0,0,8,0,9) & $A_1 \times \Dc_8$ \\
(17,2) & (41,120,80) & (16,16,7,0,2) & (0,0,1,0,6,0,10) & $A_1\times B_2$ \\
(17,3) & (41,120,80) & (18,12,7,4) & (0,0,0,0,8,0,9) & $A_1\times B_2$ \\
(17,4) & (41,120,80) & (16,16,7,0,2) & (0,0,1,0,6,0,10) & $A_1\times B_2$ \\
(17,5) & (41,120,80) & (16,18,1,6) & (0,0,0,0,6,8,1,0,2) & $A_1\times A_1\times A_1$ \\
(17,6) & (42,123,82) & (16,15,10,0,1) & (0,0,0,0,6,3,7,0,1) & $A_1\times A_1$ \\
(17,7) & (43,126,84) & (13,22,7,0,1) & (0,0,0,0,6,0,10,0,1) & $A_1\times A_1\times A_1$ \\
(17,8) & (43,126,84) & (14,20,7,2) & (0,0,0,0,1,8,8) & $A_1\times A_1\times A_1$ \\
\hline
(18,1) & (46,135,90) & (9,36,0,0,0,0,0,1) & (0,0,0,0,9,0,0,9) & $\Dc_{18}$ \\
(18,2) & (46,135,90) & (18,18,6,3,1) & (0,0,0,0,3,3,12) & $A_1\times G_2$ \\
(18,3) & (46,135,90) & (19,16,6,5) & (0,0,0,0,6,2,6,3,1) & $A_1\times A_1$ \\
(18,4) & (46,135,90) & (18,19,3,6) & (0,0,0,0,3,9,3,0,3) & $A_1\times A_2$ \\
(18,5) & (46,135,90) & (18,19,3,6) & (0,0,0,0,3,9,3,0,3) & $A_1\times A_2$ \\
(18,6) & (47,138,92) & (18,16,12,0,1) & (0,0,0,0,5,2,7,2,2) & $A_1\times A_1$ \\
(18,7) & (46,135,90) & (18,18,6,3,1) & (0,0,0,3,3,0,6,6) & $A_1\times A_2$ \\
(18,8) & (47,138,92) & (16,22,6,2,1) & (0,0,0,0,6,0,7,4,1) & $A_1\times A_1$ \\
\hline
(19,1) & (49,144,96) & (21,18,6,0,4) & (0,0,0,0,4,0,15) & $A_1\times G_2$ \\
(19,2) & (51,150,100) & (21,18,6,6) & (0,0,0,0,1,8,6,0,4) & $A_1\times A_1\times A_1$ \\
(19,3) & (49,144,96) & (24,12,6,6,1) & (0,0,0,0,4,0,15) & $A_1\times G_2$ \\
(19,4) & (51,150,100) & (20,20,6,4,1) & (0,0,0,0,4,4,4,4,3) & $A_1\times A_1\times A_1$ \\
(19,5) & (51,150,100) & (20,20,6,4,1) & (0,0,0,0,4,4,4,4,3) & $A_1\times Z_4$ \\
(19,6) & (51,150,100) & (20,20,6,4,1) & (0,0,0,0,6,0,6,4,3) & $A_1\times A_1\times A_1$ \\
(19,7) & (52,153,102) & (21,15,15,0,1) & (0,0,0,0,4,3,3,6,3) & $A_1\times A_2$ \\
\hline
(20,1) & (56,165,110) & (10,45,0,0,0,0,0,0,1) & (0,0,0,0,5,5,0,0,10) & $A_1 \times \Dc_{10}$ \\
(20,2) & (56,165,110) & (25,15,10,6) & (0,0,0,0,0,5,10,0,5) & $\Dc_{10}$ \\
(20,3) & (56,165,110) & (21,24,6,4,0,1) & (0,0,0,0,4,2,4,6,3,1) & $A_1\times A_1$ \\
(20,4) & (56,165,110) & (23,20,7,5,1) & (0,0,0,0,5,1,4,4,6) & $A_1\times A_1$ \\
(20,5) & (55,162,108) & (20,26,4,4,0,0,1) & (0,0,0,2,2,0,4,12) & $A_1\times B_2$ \\
\hline
(21,1) & (61,180,120) & (15,40,5,0,0,0,0,0,1) & (0,0,0,0,5,0,5,0,11) & $A_1 \times \Dc_{10}$ \\
(21,2) & (61,180,120) & (30,10,15,6) & (0,0,0,0,0,0,15,0,6) & $H_3$ \\
(21,3) & (61,180,120) & (24,24,9,0,4) & (0,0,0,0,6,0,3,0,12) & $B_3$ \\
(21,4) & (61,180,120) & (22,28,6,4,0,0,1) & (0,0,0,0,4,0,4,8,4,0,1) & $A_1\times B_2$ \\
(21,5) & (61,180,120) & (26,20,9,4,2) & (0,0,0,0,5,0,3,4,9) & $A_1\times A_1\times A_1$ \\
(21,6) & (63,186,124) & (25,20,15,2,1) & (0,0,0,0,1,0,11,0,8,0,1) & $A_1\times A_1\times A_1$ \\
(21,7) & (64,189,126) & (24,22,15,3) & (0,0,0,0,0,0,12,0,6,3) & $A_1\times A_2$ \\
\hline
(22,1) & (67,198,132) & (11,55,0,0,0,0,0,0,0,1) & (0,0,0,0,0,11,0,0,0,11) & $\Dc_{22}$ \\
(22,2) & (70,207,138) & (24,30,12,3,1) & (0,0,0,0,1,0,6,3,9,0,3) & $A_1\times A_2$ \\
(22,3) & (67,198,132) & (27,28,0,12) & (0,0,0,0,0,0,12,0,9,0,1) & $A_1\times A_2$ \\
(22,4) & (67,198,132) & (27,25,9,3,3) & (0,0,0,0,4,0,6,0,6,6) & $A_1\times A_2$ \\
{\bf (22,5)} & (73,216,144) & (12,58,0,0,3) & (0,0,0,0,0,0,0,12,6,0,4) & $B_3$ \\
\hline
(23,1) & (75,222,148) & (27,32,10,4,2) & (0,0,0,0,1,0,6,2,7,4,3) & $A_1\times A_1$ \\
{\bf (23,2)} & (77,228,152) & (16,56,2,0,1,2) & (0,0,0,0,0,0,1,8,10,0,4) & $A_1\times B_2$ \\
\hline
(24,1) & (79,234,156) & (12,66,0,0,0,0,0,0,0,0,1) & (0,0,0,0,0,6,6,0,0,0,12) & $A_1 \times \Dc_{12}$ \\
(24,2) & (77,228,152) & (32,32,0,12,0,0,1) & (0,0,0,0,0,4,0,0,20) & $A_1 \times \Dc_8$ \\
(24,3) & (80,237,158) & (31,32,9,5,3) & (0,0,0,0,1,0,6,1,6,6,4) & $A_1\times A_1$ \\
{\bf (24,4)} & (81,240,160) & (20,54,4,0,0,2,1) & (0,0,0,0,0,0,2,4,14,0,4) & $A_1\times B_2$ \\
\hline
(25,1) & (85,252,168) & (18,60,6,0,0,0,0,0,0,0,1) & (0,0,0,0,0,0,12,0,0,0,13) & $A_1 \times \Dc_{12}$ \\
(25,2) & (85,252,168) & (36,28,15,0,6) & (0,0,0,0,4,0,3,0,6,0,12) & $B_3$ \\
(25,3) & (91,270,180) & (30,40,15,6) & (0,0,0,0,0,0,0,0,15,0,10) & $H_3$ \\
(25,4) & (85,252,168) & (36,30,9,6,4) & (0,0,0,0,1,0,9,0,3,0,12) & $A_1\times G_2$ \\
(25,5) & (81,240,160) & (36,32,0,8,4,0,1) & (0,0,0,0,0,0,5,0,20) & $A_1 \times \Dc_8$ \\
(25,6) & (85,252,168) & (36,30,9,6,4) & (0,0,0,0,1,0,6,0,6,6,6) & $A_1\times A_2$ \\
(25,7) & (85,252,168) & (33,34,12,2,3,0,1) & (0,0,0,0,2,0,4,4,4,0,11) & $A_1\times A_1\times A_1$ \\
{\bf (25,8)} & (85,252,168) & (24,52,6,0,0,0,3) & (0,0,0,0,0,0,3,0,18,0,4) & $B_3$ \\
\hline
(26,1) & (92,273,182) & (13,78,0,0,0,0,0,0,0,0,0,1) & (0,0,0,0,0,0,13,0,0,0,0,13) & $\Dc_{26}$ \\
(26,2) & (96,285,190) & (35,40,10,11) & (0,0,0,0,0,0,0,0,11,5,10) & $\Dc_{10}$ \\
(26,3) & (92,273,182) & (37,36,9,6,3,1) & (0,0,0,0,1,0,7,2,2,1,8,4,1) & $A_1\times A_1$ \\
(26,4) & (92,273,182) & (35,39,10,4,3,0,1) & (0,0,0,0,1,1,4,4,2,2,7,4,1) & $A_1\times A_1$ \\
\hline
(27,1) & (101,300,200) & (40,40,6,14,1) & (0,0,0,0,0,0,0,0,8,8,11) & $A_1\times A_1\times A_1$ \\
(27,2) & (99,294,196) & (39,40,10,6,2,2) & (0,0,0,0,1,0,5,4,1,2,4,8,2) & $A_1\times A_1$ \\
(27,3) & (99,294,196) & (39,40,10,6,2,2) & (0,0,0,0,1,0,6,2,2,2,5,6,3) & $A_1\times A_1$ \\
(27,4) & (99,294,196) & (38,42,9,6,3,0,1) & (0,0,0,0,1,0,5,4,2,0,7,4,4) & $A_1\times A_1\times A_1$ \\
\end{longtable}
\end{tiny}

The next table contains the invariants for the unstretchable simplicial arrangements
that satisfy Pappus' theorem. Corresponding wirings are displayed in
Fig.\ \ref{nonstwirpap}.\\

\begin{tiny}
\begin{tabular}{l|llll}
lines & $(f_0,f_1,f_2)$ & $(t_2,t_3,\ldots)$ & $(r_2,r_3,\ldots)$ & $\Aut(\Ac)$ \\
\hline
16 & (38,111,74) & (12,20,3,3) & (0,0,0,0,5,7,4) & $A_1\times A_1$ \\
17 & (42,123,82) & (13,22,6,0,0,1) & (0,0,0,3,1,4,7,2) & $A_1\times A_1$  \\
18 & (46,135,90) & (14,25,6,0,0,0,1) & (0,0,0,3,1,2,8,4) & $A_1\times A_1\times A_1$ \\
19 & (55,162,108) & (15,28,12) & (0,0,0,0,1,0,12,0,6) & $A_1\times A_2$ \\
22 & (67,198,132) & (18,40,8,0,0,0,0,0,1) & (0,0,0,0,0,6,4,0,8,4) & $A_1\times A_1\times A_1$ \\
26 & (101,300,200) & (25,60,10,6) & (0,0,0,0,0,0,0,0,6,0,20) & $\Dc_{10}$
\end{tabular}
\end{tiny}

\evensidemargin8mm

\begin{figure}
\begin{center}
\includegraphics[trim = 0mm 16mm 0mm 16mm,clip,width=1.2\textwidth]{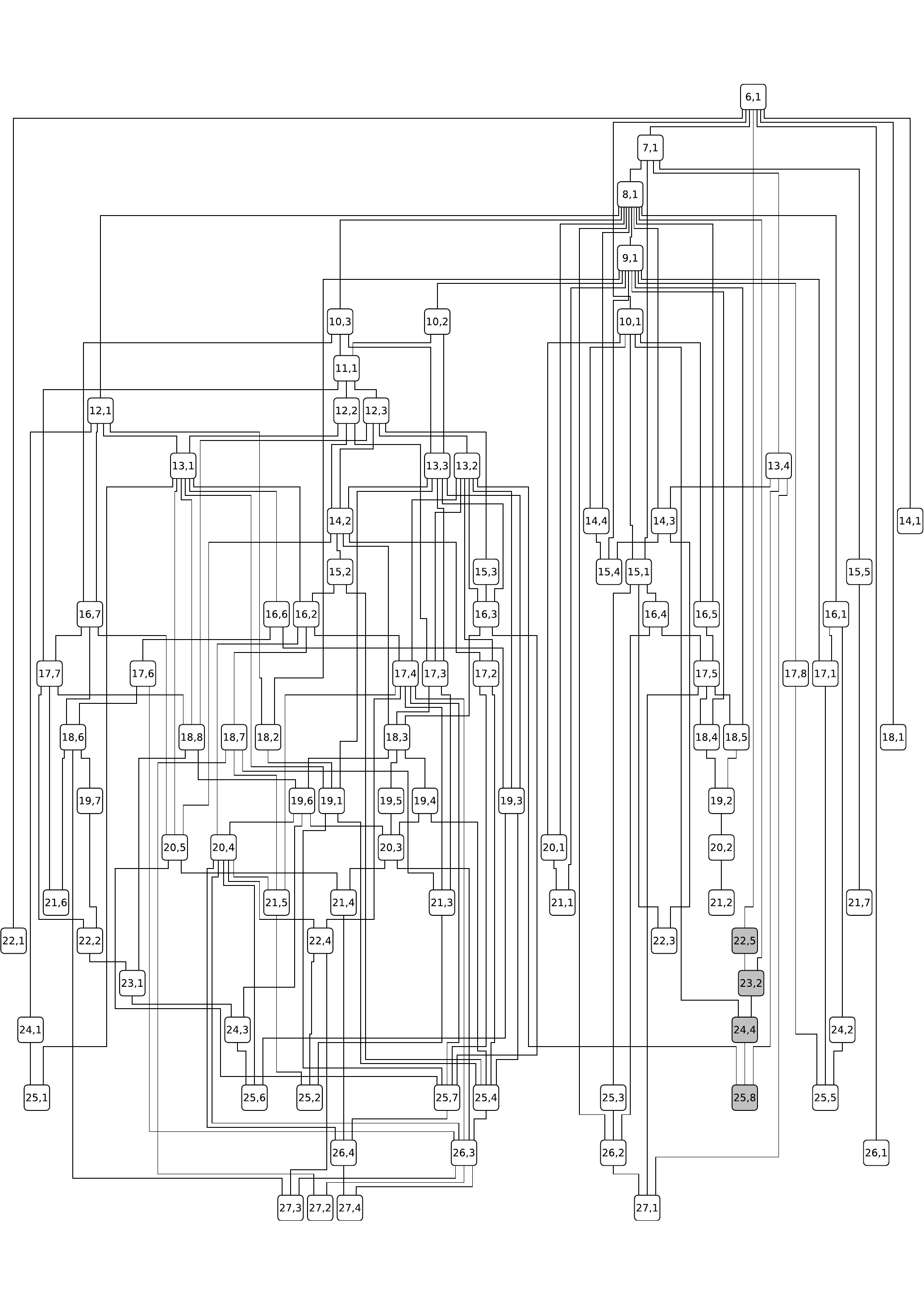}
\end{center}
\caption{A Hasse diagram (compare \cite[Fig.\ 4]{p-G-09})\label{hassediag}}
\end{figure}

\newcommand{\etalchar}[1]{$^{#1}$}
\providecommand{\bysame}{\leavevmode\hbox to3em{\hrulefill}\thinspace}
\providecommand{\MR}{\relax\ifhmode\unskip\space\fi MR }
\providecommand{\MRhref}[2]{%
  \href{http://www.ams.org/mathscinet-getitem?mr=#1}{#2}
}
\providecommand{\href}[2]{#2}


\begin{thebibliography}{BLVS{\etalchar{+}}93}

\bibitem[BLVS{\etalchar{+}}93]{BLVSWZ}
Anders Bj{\"o}rner, Michel Las~Vergnas, Bernd Sturmfels, Neil White, and
  G{\"u}nter~M. Ziegler, \emph{Oriented matroids}, Encyclopedia of Mathematics
  and its Applications, vol.~46, Cambridge University Press, Cambridge, 1993.

\bibitem[CH10a]{p-CH10}
M.~Cuntz and I.~Heckenberger, \emph{Finite {W}eyl groupoids},
  \href{http://arxiv.org/abs/1008.5291}{\texttt{arXiv:1008.5291v1}} (2010), 35
  pp.

\bibitem[CH10b]{p-CH09c}
\bysame, \emph{Finite {W}eyl groupoids of rank three}, to appear in Trans.
  Amer. Math. Soc. (2010).

\bibitem[Cun10]{p-C10b}
M.~Cuntz, \emph{Minimal fields of definition for simplicial arrangements in the
  real projective plane}, to appear in Innov. Incidence Geom. (2010).

\bibitem[Cun11]{p-C10}
\bysame, \emph{Crystallographic arrangements: Weyl groupoids and simplicial
  arrangements}, Bull. London Math. Soc. \textbf{43} (2011), no.~4, 734--744.

\bibitem[GP84]{p-GP84}
Jacob~E. Goodman and Richard Pollack, \emph{Semispaces of configurations, cell
  complexes of arrangements}, J. Combin. Theory Ser. A \textbf{37} (1984),
  no.~3, 257--293.

\bibitem[Gr{\"u}72]{b-G-72}
B.~Gr{\"u}nbaum, \emph{Arrangements and spreads}, American Mathematical
  Society, Providence, R.I., 1972, Conference Board of the Mathematical
  Sciences Regional Conference Series in Mathematics, No. 10.

\bibitem[Gr{\"u}09a]{p-G-09}
\bysame, \emph{A catalogue of simplicial arrangements in the real projective
  plane}, Ars Math.~Contemp. \textbf{2} (2009), no.~1, 25 pp.

\bibitem[Gr{\"u}09b]{p-G-09b}
\bysame, \emph{Small unstretchable simplicial arrangements of pseudolines},
  Geombinatorics \textbf{18} (2009), 153--160.

\bibitem[OT92]{OT}
P.~Orlik and H.~Terao, \emph{Arrangements of hyperplanes}, Grundlehren der
  Mathematischen Wissenschaften, vol. 300, Springer-Verlag, Berlin, 1992.

\end{thebibliography}
\end{document}